\documentclass[useAMS,referee,usenatbib]{biom}

\renewcommand\arraystretch{0.8}
\renewcommand{\theequation}{\thesection.\arabic{equation}}

\usepackage{amsmath, amssymb, amsfonts, amsbsy, epsfig, graphicx,rotating, subfigure}
\usepackage[usenames, dvipsnames]{color}
\usepackage{enumerate}
\RequirePackage[colorlinks,citecolor=blue,urlcolor=red]{hyperref}
\usepackage{natbib}
\usepackage{ulem}
\usepackage{fancyhdr}
\usepackage{multirow,xspace,setspace,booktabs,subfigure}
\usepackage{tabularx}
\usepackage{tabulary}

\def\be{\begin{equation}}
\def\ee{\end{equation}}

\newcommand{\wh}{\widehat}

\def\ga{\gamma}
\def\de{\delta}

\def\la{\lambda}

\newcommand{\wt}{\widetilde}

\renewcommand{\P}{\mathbb{P}}

\newcommand{\s}{\sqrt}

\newcommand{\br}{\mathbb{R}}

\newcommand{\bA}{{\mathbf A}}

\newcommand{\bF}{{\mathbf F}}

\newcommand{\bR}{{\mathbf R}}

\newcommand{\bU}{{\mathbf U}}
\newcommand{\bV}{{\mathbf V}}
\newcommand{\bW}{{\mathbf W}}
\newcommand{\bX}{{\mathbf X}}
\newcommand{\bY}{{\mathbf Y}}
\newcommand{\bZ}{{\mathbf Z}}

\newcommand{\bbeta}  {\boldsymbol{\beta}}

\newcommand{\bSigma}{\boldsymbol{\Sigma}}

\newcommand{\bTheta} {\boldsymbol{\Theta}}

\newcommand{\bmu} {\boldsymbol{\mu}}

\newcommand{\bD}{{\mathbf D}}
\newcommand{\bzero}{{\mathbf 0}}

\newcommand{\lbl}{\label}
\newcommand{\beq}{\begin{eqnarray*}}
\newcommand{\eeq}{\end{eqnarray*}}
\newcommand{\beqn}{\begin{eqnarray}}
\newcommand{\eeqn}{\end{eqnarray}}
\newcommand{\eq}[1]{$(\ref{#1})$}
\newcommand{\rI}{\mathrm{\scriptscriptstyle (I)}}
\newcommand{\rII}{\mathrm{\scriptscriptstyle (II)}}
\newcommand{\rs}{{\rm s}}
\newcommand{\ns}{{\rm ns}}

\def\T{{ \mathrm{\scriptscriptstyle T} }}

\def\JRSSB{{\it Journal of the Royal Statistical Society, Series B}}

\def\BIN{{\it Bioinformatics}}
\def\SS{{\it Statistica Sinica}}

\def\AS{{\it The Annals of Statistics}}

\title[Testing High Dimensional Means]{Simulation-Based Hypothesis Testing of High Dimensional Means Under Covariance Heterogeneity}

\author{Jinyuan Chang$^{1,*}$\email{changjinyuan@swufe.edu.cn},
Chao Zheng$^{2,**}$\email{zhengc1@student.unimelb.edu.au}, Wen-Xin Zhou$^{3,***}$\email{wenxinz@princeton.edu}, and Wen Zhou$^{4,****}$\email{riczw@stat.colostate.edu}  \\
\small $^{1}$School of Statistics, Southwestern University of Finance and Economics, Chengdu, Sichuan 611130, China\\
\small $^{2}$School of Mathematics and Statistics, The University of Melbourne, Parkville, VIC 3010, Australia \\
\small $^{3}$Department of Operations Research and Financial Engineering, Princeton University, Princeton, NJ 08544, U.S.A.\\
\small $^{4}$Department of Statistics, Colorado State University, Fort Collins, CO 80523, U.S.A.
}

\begin{document}


\date{{\it Received December} 2016.  }



\pagerange{\pageref{firstpage}--\pageref{lastpage}}
\volume{}
\pubyear{2015}
\artmonth{September}


\doi{10.1111/j.1541-0420.2005.00454.x}


\label{firstpage}


\begin{abstract}
In this paper, we study the problem of testing the mean vectors of high dimensional data in both one-sample and two-sample cases. The proposed testing procedures employ maximum-type statistics and the parametric bootstrap techniques to compute the critical values. Different from the existing tests that heavily rely on the structural conditions on the unknown covariance matrices, the proposed tests allow general covariance structures of the data and therefore enjoy wide scope of applicability in practice. To enhance powers of the tests against sparse alternatives, we further propose two-step procedures with a preliminary feature screening step. Theoretical properties of the proposed tests are investigated. Through extensive numerical experiments on synthetic datasets and an human acute lymphoblastic leukemia gene expression dataset, we illustrate the performance of the new tests and how they may provide assistance on detecting disease-associated gene-sets. The proposed methods have been implemented in an R-package HDtest and are available on CRAN.
\end{abstract}


\maketitle

\section{Introduction}   \lbl{sec.intro}


The problems of comparing a particular sample to a hypothetical population with known prior information or comparing two parallel groups, such as a control group and a treatment group, have both important applications in modern genomics and bio-medical research and become the foundation of scientific discoveries. They have been employed widely for identifying biologically interesting gene-sets for drug design, evolutionary studies, and mutation detection. Our interests in these problems are motivated by a microarray study on human acute lymphoblastic leukemia  \citep{C04}. This study consists of 75 patients of B-lymphocyte type leukemia, who were classified into two groups: 35 patients with BCR/ABL fusion and 40 patients with cytogenetically normal NEG. It is known that genes tend to work collectively
in groups to achieve certain biological tasks. Our analysis focuses on such groups of genes (gene sets) defined with the gene ontology (GO) framework, which are referred to as GO terms. Identifying disease-relevant GO terms based on their average expression levels provides information on differential gene pathways associated with the leukemia. Many GO terms contain a large number of (in the data, as many as 3,145) genes with very complex gene-wise dependence structures. The large dimension of data and the complex dependency among genes make the problem of comparing population means extremely challenging.

 Let $\bX$ and $\bY$ be two $p$-dimensional random vectors with means $\bmu_1=(\mu_{1 1}, \ldots, \mu_{1 p})^\T$ and $\bmu_2=(\mu_{2 1}, \ldots, \mu_{2p})^\T$, covariance matrices $\bSigma_1=(\sigma_{1, k \ell})_{1\leq k, \ell \leq p}$ and $\bSigma_2=(\sigma_{2, k \ell})_{1\leq k, \ell \leq p}$, respectively. It is then of general interest in testing the hypotheses
\begin{itemize}
\item (One-sample problem)
$
H^{\rI}_{0} : \bmu_{1}=\bmu_0$ versus $H^{\rI}_{ 1}:\bmu_1\neq\bmu_0
$
for a   specified $p$-dimensional vector $\bmu_0$, which, without loss of generality, is equivalent to
\begin{equation}
H^{\rI}_{0} : \bmu_{1}=\bzero ~~ {\rm versus} ~~ H^{\rI}_{ 1}:\bmu_1\neq\bzero; \label{eq:onesample}
\end{equation}
 \item (Two-sample problem)
\begin{equation}\label{eq:twosample}
H^{\rII}_0:\bmu_1=\bmu_2 ~~ {\rm versus} ~~ H^{\rII}_1:\bmu_1\neq\bmu_2.
\end{equation}
\end{itemize}
 When $p$ is fixed, traditional tests have been extensively studied for testing both \eqref{eq:onesample} and \eqref{eq:twosample}. For example, the properties for both the one-sample and two-sample Hotelling's $T^2$ tests have been examined under normality assumption \citep{Anderson_2003}. We refer to \cite{LiuShao_2013} for a moderate deviation result in the absence of normality.

Generally, the sum of squares-type and the maximum-type statistics are used to test the hypotheses \eq{eq:onesample} and \eq{eq:twosample} in the high dimensional settings. The sum of squares-type statistics aim to mimic the weighted Euclidean norms, $| \bA \bmu_1 |_2^2 $ or $|\bA (\bmu_1-\bmu_2) |_2^2$ for certain linear transformation $\bA$, and the corresponding tests are powerful for detecting relatively dense signals \citep{BaiSaranadasa_1996, ChenQin_2010}. Statistics of the maximum-type, on the other hand, are preferable for detecting relatively sparse signals \citep{CaiLiuXia_2014} and have been used in a variety of applications including the medical image problem \citep{JamesClymerSchmalbrock_2001} 
and gene selections \citep{Martens_2005}. 

Most existing testing procedures for \eq{eq:onesample} and \eq{eq:twosample} rely on the derivation the pivotal limiting distribution of test statistics, from which the critical value is approximated. In the high dimensional scenarios, various structural assumptions on the unknown covariance matrices have been imposed \citep{ZhongChenXu_2013,CaiLiuXia_2014}. 
However, in many applications, these assumptions can be very restrictive or difficult to be verified, and therefore limit the scope of applicability for the limiting distribution calibration approach. First, the existence of a pivotal asymptotic distribution relies heavily on the structural assumptions on the unknown covariance/correlation structures, which may not be true in practice. For example, it is very common that the expression levels are highly correlated for genes regulated by the same pathway \citep{WM2012} or associated with the same functionality \citep{K2014}, which results in a complex and non-sparse covariance structure. These empirical evidences indicate that the strong structural assumptions on the covariance matrices may sometimes be unrealistic in real-world applications. Another concern, as pointed out by \cite{CaiLiuXia_2014}, is that the convergence rate to the extreme value distribution of maximum-type statistics is usually slow. Taking the extreme distribution of type I as an example, the convergence rate is of order $O\{  \log(\log n)/ \log(n) \}$. Although the convergence rate may be improved by using suitable intermediate approximations, still its validity relies on the dependence structure of the underlying distribution.

Driven by the above two concerns, we revisit the problem of testing hypotheses \eq{eq:onesample} and \eq{eq:twosample} from a different perspective. Motivated by applications in genomic analysis and image analysis, we are particularly interested in detecting discrepancies when $\bmu_1$ and $\bzero$ or $\bmu_2$ are distinguishable to a certain extent in at least one coordinate. We develop a fully data driven procedure to compute the critical values using the Monte Carlo simulations. The validity of our procedure is established without enforcing structural assumptions of any kind on the unknown covariances. The main idea is based on the approximation of empirical processes by Gaussian processes \citep{ChernozhukovChetverikovKato_2013}, and to some degree, is similar to that of \cite{LiuShao_2013} that utilizes the intermediate approximation. However, instead of generating independent standard multivariate normal vectors, our approach takes into account correlations among the features and therefore is automatically adapted to the underlying dependence. 


The rest of the paper is organized as follows. In Section~\ref{method.sec}, we describe the simulation-based testing procedures for both hypotheses \eqref{eq:onesample} and \eqref{eq:twosample}. Theoretical properties of the tests are studied in Section~\ref{section:theory}. Numerical studies are reported in Section \ref{section:numerical} to assess the performance of the proposed tests comparing to the peer methods. In Section \ref{real data}, we applied the proposed tests to the acute lymphoblastic leukemia data for identifying disease-associated gene-sets based on the gene expression levels. The underpinning technical details, as well as additional simulation results and empirical data analysis, are relegated to the supplementary material.

\setcounter{equation}{0}
\section{Methodology}
\label{method.sec}

Throughout the paper, we denote by $|\bbeta|_\infty=\max_{1\leq k\leq p}|\beta_k|$ for a $p$-dimensional vector $\bbeta=(\beta_1,\ldots,\beta_p)^\T$. For a matrix $\bA=(a_{k \ell})_{p\times p}$, define $|\bA|_\infty=\max_{1\leq k , \ell \leq p}|a_{k  \ell}|$. 
Let $\bD_1=\mbox{diag}\,(\bSigma_1)$ and $\bD_2=\mbox{diag}\,(\bSigma_2)$. Denote by $\bR_1$ and $\bR_2$ the corresponding correlation matrices. Let $\mathcal{X}_n=\{\bX_1, \ldots , \bX_{n} \}$ and $\mathcal{Y}_m=\{\bY_1, \ldots, \bY_{m}\}$ be two independent samples consisting of independent and identically distributed (i.i.d.) observations drawn from the distributions of $\bX$ and $\bY$, respectively. Let $N=n+m$. For each $i=1,\ldots,n$ and $j=1,\ldots,m$, write $\bX_i=(X_{i1},\ldots,X_{ip})^\T$ and $\bY_j=(Y_{j1},\ldots,Y_{jp})^\T$.

\subsection{Test procedures}

\subsubsection{One-sample case}
\label{section:onesample}
Consider the maximum-type statistics in the following forms:
\be
 T_{\ns}^{\rI}=\max_{1\leq k\leq p} \s{n} |\bar{X}_k| \qquad \textrm{or}  \qquad
 T^{\rI}_{\rs}=\max_{1\leq k\leq p}\frac{\s{n}|\bar{X}_k|}{\hat{\sigma}_{1 k}},   \label{T12}
\ee where $\bar{X}_k=n^{-1}\sum_{i=1}^{n}X_{i k}$ and $\hat{\sigma}_{1 k}^2=n^{-1}\sum_{i=1}^{n}(X_{i k}-\bar{X}_k)^2$. Throughout, the statistic $T^{\rI}_{\rs}$ is referred as the {\it studentized} statistic, while $T_{\ns}^{\rII}$ is referred as the {\it non-studentized} statistic. Intuitively, large values of $T^{\rI}_{\ns}$ or $T^{\rI}_{\rs}$ provide evidences against $H^{\rI}_0$ in \eqref{eq:onesample} so that the corresponding tests are of the form $
    \Psi_{\ns,\alpha}^{\rI}=I\{T_{\ns}^{\rI} >  {\rm cv}_{\ns, \alpha}^{\rI}\}$ or $\Psi_{\rs,\alpha}^{\rI}=I\{T_{\rs}^{\rI} >  {\rm cv}_{\rs , \alpha}^{\rI}\} ,
$ where ${\rm cv}_{\ns, \alpha}^{\rI} $ and ${\rm cv}_{\rs , \alpha}^{\rI} $ are the critical values.

Under the null hypothesis $H^{\rI}_0:\bmu_1=\bzero$, we motivate from the multivariate central limit theorem with fixed $p$ to calculate critical values ${\rm cv}_{\ns,\alpha}^{\rI}$ and ${\rm cv}_{\rs ,\alpha}^{\rI}$ as follows:
let $\wt{\bSigma}_1$ be an estimate of $\bSigma_1$ from the sample $\mathcal{X}_n$, and set $\wt{\bR}_1= \wt{\bD}_1^{-1/2}\wt{\bSigma}_1 \wt{\bD}_1^{-1/2}$ with $\wt{\bD}_1=\textrm{diag}\,(\wt{\bSigma}_1)$. Given $\mathcal{X}_n$, let $\bW^{\rI}_{\ns} \sim \text{N}(\bzero,\wt{\bSigma}_1)$ and $\bW^{\rI}_{\rs} \sim \text{N}(\bzero, \wt{\bR}_1)$ be two Gaussian random vectors, the critical values  can be computed by
$
 {\rm cv}_{\ns,\alpha}^{\rI}  =~ \inf\{  t\in \br: \mathbb{P} (  |\bW^{\rI}_{\ns}|_\infty >t \,|   \mathcal{X}_n ) \leq \alpha\}$ and $ 
  {\rm cv}_{\rs ,\alpha}^{\rI}   =~ \inf\{ t\in \br: \mathbb{P} (  |\bW^{\rI}_{\rs} |_\infty >t \, | \mathcal{X}_n ) \leq \alpha \}.  
$
Practically, let $ \{\bW_{\ns,\ell}\}_{\ell=1}^M\overset{\rm i.i.d.}{\sim}\text{N}(\bzero,\wt{\bSigma}_1)$ and $\{\bW_{{\rs},\ell}\}_{\ell=1}^M\overset{\rm i.i.d.}{\sim}\text{N}(\bzero, \wt{\bR}_1)$. Then, ${\rm cv}_{\ns,\alpha}^{\rI}$ and ${\rm cv}_{\rs ,\alpha}^{\rI}$ can be estimated by
$
\widehat{{\rm cv}}_{\ns,\alpha}^{\, \rI}  = \inf \{  t\in \br:   \wh{F}_{\ns,M}^{\rI}(t) \geq 1- \alpha \}$ and $
 \widehat{{\rm cv}}_{\rs ,\alpha}^{\,\rI} = \inf\{  t\in \br:   \wh{F}_{{\rs},M}^{\rI}(t) \geq 1- \alpha \},
$
where $\wh{F}_{\ns,M}^{\rI}(t) =  M^{-1} \sum_{\ell=1}^M I\{|\bW_{\ns,\ell}|_\infty \leq t \} $ and $\wh{F}_{{\rs},M}^{\rI}(t) = M^{-1}\sum_{\ell=1}^M I\{|\bW_{\rs ,\ell}|_\infty \leq t \} $.
For $\nu\in\{\ns ,  {\rs}\}$, the empirical version of test $\Psi_{\nu,\alpha}^{\rI}$ is therefore defined by \be
     \wh \Psi_{\nu, \alpha}^{\rI}(M) = I\{ T_{\nu}^{\rI} > \widehat{{\rm cv}}_{\nu, \alpha}^{\,\rI} \},  \label{test.12}
\ee such that the null hypothesis $H^{\rI}_0$ is rejected whenever $\wh \Psi^{\rI}_{\nu, \alpha}(M) =1 $. The proposed testing procedures are fully data driven and easily computed.
In Section \ref{section:matrixest}, we discuss the constructions of $\wt{\bSigma}_1$, from which the wide applicability of the test \eqref{test.12} will be explored. 
\subsubsection{Two-sample case}
\label{2s.test}

The above procedures can be naturally extended to deal with the two-sample problem \eqref{eq:twosample}. Analogously to \eq{T12}, we define the non-studentized and studentized test statistics by
$
 T_{\ns}^{\rII}=\max_{1\leq k\leq p}  \s{nm} |\bar{X}_k-\bar{Y}_k|/\sqrt{n+m} $ and $T_{\rs}^{\rII} = \max_{1\leq k\leq p}   \s{nm}  |\bar{X}_k-\bar{Y}_k|/(m \hat{\sigma}_{1 k}^2   + n \hat{\sigma}_{2 k}^2)^{1/2}$
respectively, where $\bar{X}_k=n^{-1}\sum_{i=1}^{n}X_{i k}$, $\bar{Y}_k=m^{-1}\sum_{j=1}^{m}Y_{j k}$, $\hat{\sigma}_{1 k}^2=n^{-1}\sum_{i=1}^{n}(X_{i k}-\bar{X}_k)^2$, and $ \hat{\sigma}_{2 k}^2=m^{-1} \sum_{j=1}^{m}(Y_{j k}-\bar{Y}_k)^2.$ 
For nominal significance level $\alpha$, we define tests of the form
$
    \Psi_{\ns,\alpha}^{\rII}=I\{T_{\ns}^{\rII} > {\rm cv}_{\ns,\alpha}^{\rII}\}$ or $\Psi_{\rs,\alpha}^{\rII}=I\{T_{\rs}^{\rII} > {\rm cv}_{\rs ,\alpha}^{\rII}\}$
    with appropriate critical values ${\rm cv}_{\ns,\alpha}^{\rII}$ and ${\rm cv}_{\rs ,\alpha}^{\rII}$. Let $\wt{\bSigma}_1$ and $\wt{\bSigma}_2$ be estimates of $\bSigma_1$ and $\bSigma_2$, respectively. Define
\be
    \wt{\bSigma}_{1,2} = \frac{m}{N}\wt{\bSigma}_1+ \frac{n}{N} \wt{\bSigma}_2,  \quad   \wt{\bD}_{1,2}=\textrm{diag}\,\big( \wt{\bSigma}_{1,2}\big), \quad  \wt{\bR}_{1,2} = \wt{\bD}_{1,2}^{-1/2} \wt{\bSigma}_{1,2} \wt{\bD}_{1,2}^{-1/2}, \label{Sigma12} \ee
and let $\{\bW_{\ns,\ell}\}_{\ell=1}^M\overset{\rm i.i.d.}{\sim}\text{N}(\bzero, \wt{\bSigma}_{1,2} )$ and $\{\bW_{{\rs},\ell}\}_{\ell=1}^M\overset{\rm i.i.d.}{\sim}\text{N}(\bzero, \wt{\bR}_{1,2}  )$. Then, ${\rm cv}_{\ns,\alpha}^{\rII}$ and ${\rm cv}_{\rs ,\alpha}^{\rII}$ can be estimated by
$
  \widehat{{\rm cv}}_{\ns,\alpha}^{\,\rII} =  \inf \{  t\in \br: \wh{F}_{\ns,M}^{\rII}(t) \geq 1- \alpha \}$ and $  \widehat{{\rm cv}}_{\rs ,\alpha}^{\, \rII}  =  \inf \{  t\in \br:   \wh{F}_{{\rs},M}^{\rII}(t) \geq 1- \alpha  \},
$
where $\wh{F}_{\ns,M}^{\rII}(t) =  M^{-1} \sum_{\ell=1}^M I\{|\bW_{\ns,\ell}|_\infty \leq t \} $ and $\wh{F}_{{\rs},M}^{\rII}(t) = M^{-1} \sum_{\ell=1}^M I\{|\bW_{\rs ,\ell}|_\infty \leq t \} $. Similarly to \eq{test.12}, for $\nu \in\{\ns, {\rs}\}$, we define the empirical version of $\Psi^{\rII}_{\nu, \alpha}$ by $ \wh \Psi^{\rII}_{\nu, \alpha}(M)= I \{ T^{\rII}_{\nu} > \widehat{{\rm cv}}_{\nu, \alpha}^{\,\rII} \}$, such that the null hypothesis $H^{\rII}_0$ is rejected as long as $\wh \Psi^{\rII}_{\nu, \alpha}(M)=1$.

\subsection{Estimation of covariance matrices}
\label{section:matrixest}

As a part of proposed tests, we need estimates of the covariance matrices. Many existing tests rely on the operator-norm consistent estimation of the covariance matrices that requires extra structural assumptions on the unknown covariances such as banding or sparsity. 
In contrast, the proposed tests require much less restrictions on covariance estimates, which grants its wide scope of applicability. In fact, the validity of the proposed testing procedures only entails the covariance estimators $\wt{\bSigma}_1$ and $\wt{\bSigma}_2$ to satisfy $|\wt{\bSigma}_1-\bSigma_1|_\infty=o_P(1)$ and $|\wt{\bSigma}_2-\bSigma_2|_\infty=o_P(1)$. 


It is shown in Lemma 3 in the supplementary material that for the sample covariance and correlation matrices $\widehat{\bSigma}_q$ and $\wh{\bR}_q$ with $q=1,2$, there holds $
|\widehat{\bSigma}_q-\bSigma_q|_\infty + |\widehat{\bR}_q-\bR_q|_\infty  =o_P(1)$ under mild regularity conditions for $\log(p) = o( n^{ \gamma/2 })$ with $0<\gamma\leq 2$.
Therefore, the sample covariance and correlation matrices can be directly used in the proposed tests, while the dimension $p$ is allowed to be as large as either $O\{ \exp(n^{c_1}) \}$ for some $c_1>0$. 
In comparison to the existing tests,  we do not enforce any structural assumptions on the unknown covariance matrices $\bSigma_1$ and $\bSigma_2$. This reflects our motivations in Section \ref{sec.intro}. As evidenced by extensive numerical studies in Section \ref{section:numerical}, our proposed procedures are fairly robust to various covariance structures with complex forms, even the long range dependence.
Although the proposed tests do not require operator-norm consistent estimates of $\bSigma_1$ and $\bSigma_2$, still one may replace the sample covariance matrix by adaptive and rate-optimal covariance estimators to improve the empirical performance when the underlying covariance satisfies certain structural assumptions. 


\subsection{Screening-based testing procedures}
\label{sec:ps}

The proposed testing procedures are valid when the dimension $p$ is much larger than the sample size $n$. However, building tests based on all dimensions may result in large critical values which may compromise the power performance. To enhance the power, we propose a two-step procedure that combines the proposed simulation-based tests and a preliminary step on {\it feature screening}, which screens the $p$ measurements before conducting the test. The power of this two-step procedure is expected to improve upon the proposed tests with a large number of irrelevant features excluded. 



\subsubsection{One-sample case}
\label{ps:sec1}
Let $\mathcal{S}_{10}=\{1\leq k\leq p:\mu_{1 k}=0\}$. The preliminary procedure is aimed at eliminating irrelevant features indexed by $\mathcal{S}_{10}$. Reformulate the original global test of a mean vector to the following $p$ marginal tests:
$
H^{\rI}_{0k}:\mu_{1 k}=0 $ versus $H^{\rI}_{1k}:\mu_{1 k}\neq 0,
$
for $k=1, \ldots, p$. For the $k$th marginal hypothesis, a standard test statistic is the $t$-statistic $
\textrm{TS}_k^{\rI}={ \s{n} |\bar{X}_k|}/{\hat{\sigma}_{1 k}}.$
Motivated by the idea of marginal screening \citep{ChangTangWu_2013,ChangTangWu_2016}, we define the index set
$
\widehat{\mathcal{S}}_{1 } =\{ 1\leq k\leq p:  \textrm{TS}_k^{\rI} \leq  \sqrt{2\log (p)}+\{2\log (p)\}^{-1/2}+\sqrt{2\log(1/\alpha)}\}.
$
We refer to \cite{ChangTangWu_2013,ChangTangWu_2016} for more discussions on the advantages of the studenized statistics in marginal screening problems. If $|\widehat{\mathcal{S}}_1|<p$, we put $d=p-|\widehat{\mathcal{S}}_{1 }|$ and let $\wt{\bmu}_1 \in \br^d$ be the sub-vector of $\bmu_1 \in \br^p$ containing only the coordinates excluded by $\widehat{\mathcal{S}}_{1}$. We have therefore downsized the original problem and instead, we focus on the reduced null hypothesis $\wt{H}^{\rI}_0: \wt{\bmu}_1=\bzero$ against the alternative $\wt{H}_1^{\rI}: \wt{\bmu}_1\neq\bzero$. Write $\widehat{T}_{\textrm{ns}}^{(\textrm{I})}=\max_{k\notin\widehat{\mathcal{S}}_1}\sqrt{n}|\bar{X}_k|$ and $\widehat{T}_{\textrm{s}}^{(\textrm{I})}=\max_{k\notin\widehat{\mathcal{S}}_1}\sqrt{n}|\bar{X}_k|/\hat{\sigma}_{1k}$. The resulting non-studentized and studentized tests are given by
$\Psi_{\ns,\alpha}^{f,\rI} = I\{ \widehat{T}_{\textrm{ns}}^{(\textrm{I})} > {\rm cv}^{\rI}_{\ns, \alpha}(\widehat{\mathcal{S}}_{1}) \}$ and $\Psi^{f,\rI}_{\rs,\alpha}=I\{ \widehat{T}_{\textrm{s}}^{(\textrm{I})} > {\rm cv}^{\rI}_{\rs, \alpha}(\widehat{\mathcal{S}}_{1})  \}$,
where ${\rm cv}^{\rI}_{\ns, \alpha}(\widehat{\mathcal{S}}_{1 })$ and ${\rm cv}^{\rI}_{\rs, \alpha}(\widehat{\mathcal{S}}_{1 })$ denote the conditional $(1-\alpha)$-quantile of $\max_{k\notin \widehat{\mathcal{S}}_{1 }} | W_{\ns, k}^{\rI}|$ and $\max_{k\notin \widehat{\mathcal{S}}_{1 }} | W_{\rs, k}^{\rI}|$ given $\mathcal{X}_n$, respectively, with $\bW^{\rI}_{\ns}=(W^{\rI}_{\ns,1}, \ldots, W^{\rI}_{\ns,p})^\T$ and $\bW^{\rI}_{\rs}=(W^{\rI}_{\rs,1}, \ldots, W^{\rI}_{\rs,p})^\T$ as discussed in Section \ref{section:onesample}. Whenever $|\widehat{\mathcal{S}}_1|=p$, we set $\Psi_{\ns,\alpha}^{f,\rI} = \Psi_{\rs,\alpha}^{f,\rI}=0$.

Notice that $\mathbb{P}_{H_0^{(\textrm{I})}}\{\Psi_{\nu,\alpha}^{f,(\textrm{I})}=1\}\leq\mathbb{P}_{H_0^{(\textrm{I})}}[\Psi_{\nu,\alpha}^{f,(\textrm{I})}=1,\widehat{\mathcal{S}}_1=\{1,\ldots,p\}]+\mathbb{P}_{H_0^{(\textrm{I})}}[\widehat{\mathcal{S}}_1\neq \{1,\ldots,p\}]$ for $\nu\in\{\textrm{ns},\textrm{s}\}$. Since $\Psi_{\nu,\alpha}^{f,(\textrm{I})}=0$ if $|\widehat{\mathcal{S}}|=p$, then $\mathbb{P}_{H_0^{(\textrm{I})}}\{\Psi_{\nu,\alpha}^{f,(\textrm{I})}=1\}\leq\mathbb{P}_{H_0^{(\textrm{I})}}[\widehat{\mathcal{S}}_1\neq \{1,\ldots,p\}]$. As shown in part D of supplementary material, $\lim\sup_{n\rightarrow\infty}\mathbb{P}_{H_0^{(\textrm{I})}}[\widehat{\mathcal{S}}_1\neq \{1,\ldots,p\}]\leq \alpha$, which indicates that the size of the two-step procedure can be controlled by the prescribed significant level $\alpha$. On the other hand, also stated in part D of supplementary material, $\mathbb{P}_{H_1^{(\textrm{I})}}\{\widehat{T}_{\nu}^{(\textrm{I})}=T_{\nu}^{(\textrm{I})}\}\rightarrow1$ for $\nu\in\{\textrm{ns},\textrm{s}\}$ which means the testing statistics with screening and without screening are almost identical under $H_1^{(\textrm{I})}$. Since the critical value $\textrm{cv}_{\nu,\alpha}^{(\textrm{I})}(\widehat{\mathcal{S}}_1)$ for two-step procedure is not larger than $\textrm{cv}_{\nu,\alpha}^{(\textrm{I})}$ for non-screening procedure, we know with probability approaching to one that the power for two-step procedure does not decrease in comparison to the procedure without screening. The simulation studies in Section \ref{section:numerical} also verify this.

\subsubsection{Two-sample case}
\label{ps:sec2}
Similar to the one-sample case, for each $k=1,\ldots,p$, we define $\textrm{TS}_k^{\rII}=  \sqrt{nm}|\bar{X}_k-\bar{Y}_k|/ \big( m\hat{\sigma}_{1 k}^2  + n\hat{\sigma}_{2 k}^2 \big)^{1/2}$ and set
$
\widehat{\mathcal{S}}_{2 }=\{1\leq k\leq p: \textrm{TS}_k^{\rII}  \leq   [\sqrt{2\log (p)}+ \{2\log(p) \}^{-1/2}+\sqrt{2\log(1/\alpha)}\}.  $
If $|\widehat{\mathcal{S}}_2|<p$, the resulting tests, denoted by $\Psi_{\ns, \alpha}^{f,\rII}$ and $\Psi_{\rs, \alpha}^{f,\rII}$, are defined in the same way as $\Psi_{\ns, \alpha}^{f,\rI}$ and $\Psi_{\rs, \alpha}^{f,\rI}$ for one-sample case respectively. If $|\widehat{\mathcal{S}}_2|=p$, we set $\Psi_{\ns,\alpha}^{f,\rII} = \Psi_{\rs,\alpha}^{f,\rII}=0$.

\setcounter{equation}{0}
\section{Theoretical properties}\label{section:theory}

In this section, we study the properties of the proposed tests including the asymptotic sizes and powers.
In practice, taking $M$ in thousands using numerical devices to increase simulation efficiency is now the rule rather than the exception in the Monte Carlo framework. The difference between such large values of $M$ and using mathematically ideal value $M=\infty$ is particularly small. We therefore focus on the oracle tests $\Psi_{\nu,\alpha}^{\rI}$ and $\Psi_{\nu,\alpha}^{\rII}$ for $\nu\in\{\ns,\rs\}$, and their screening-based analogues $\Psi_{\nu,\alpha}^{f,\rI}$ and $\Psi_{\nu,\alpha}^{f,\rII}$. It is shown that the proposed tests maintain the nominal size asymptotically under very general covariance structures. Moreover, the proposed tests are shown to be consistent against sparse alternatives. Recall $\bSigma_1=(\sigma_{1,k\ell})_{1\leq k,\ell\leq p}$, $\bSigma_2=(\sigma_{2,k\ell})_{1\leq k,\ell\leq p}$, $\bD_1=\mbox{diag}\,(\bSigma_1)$ and $\bD_2=\mbox{diag}\,(\bSigma_2)$. The marginally standardized version of $\bX$ and $\bY$ are $\bU=(U_{1},\ldots, U_p)^\T=\bD_1^{-1/2}\bX$ and $\bV=(V_1,\ldots, V_p)^\T=\bD_2^{-1/2}\bY$, respectively.
We only impose the following mild moment conditions.  \\
\noindent
\textbf{(M1)} $\max_{1\leq k\leq p}  \max[  \{\mathbb{E} (|U_k|^r  )\}^{1/r} ,  \{\mathbb{E} (|V_k|^r )\}^{1/r}  ] \leq K_0$ for some $r \geq 4$ and $K_0>0$\\
\noindent
\textbf{(M2)} $\max_{1\leq k\leq p} \max[ \mathbb{E}  \{\exp(K_1|U_k|^{\gamma} )\} , \mathbb{E}  \{\exp(K_1 |V_k|^{\gamma} )\} ] \leq K_2$ for some $K_1>0$, $K_2>1$ and $0<\gamma\leq 2$.

Condition~(M1) indicates that the tail probability $\P(|U_k|  > t )$ decays to zero in a faster rate than $t^{-r}$ as $t\to \infty$. Condition~(M2) requires exponentially light tails, i.e., $\P(|U_k|  > t ) \leq \exp(- \wt K_1 t^\gamma)$ for some $\wt K_1>0$ and all sufficiently large $t$, and implies that all moments of $U_k$ are finite. Throughout this section, we assume that $\sigma_{1,11},\ldots,\sigma_{1,pp},\sigma_{2,11},\ldots,\sigma_{2,pp}$ are uniformly bounded away from $0$ and $\infty$, $n, p\geq 2$, $n \asymp m$ and $n\leq m$.

\begin{theorem}  \label{one-sample.size}
Let $\wt{\bSigma}_1=\wh{\bSigma}_1$, the sample covariance matrix, and $\nu \in \{\ns, \rs\}$. As $n ,p  \rightarrow \infty$,
$\mathbb{P}_{H^{\rI}_0} \{  \Psi^{\rI}_{\nu,\alpha} =1 \} \rightarrow \alpha
$ holds with either {\rm(i)} {\rm(M1)} holds and $p=O(n^{r/2-1-\de})$ for some $\de>0$; or {\rm(ii)} {\rm(M2)} holds for some $\gamma\geq 1/2$ and $\log (p)=o(n^{1/7})$.
\end{theorem}

Theorem \ref{one-sample.size} establishes the validity of the proposed one-sample tests in the sense that the testing procedures in Section \ref{section:onesample} maintain nominal significance level asymptotically. In addition, as evidenced by the numerical experiments in Section~\ref{section:numerical}, the test based on non-studentized statistics outperforms its studentized analogue in terms of maintaining the nominal significance level when the sample size is small. This, however, is not surprising since the inverse operation, say $\wh{\bD}_1^{-1/2}$, usually leads to an augmentation of the estimation error in $\wh{\bD}_1$ and therefore is more sensitive to the sample size. In the following theorem, we summarize the asymptotic power of the proposed one-sample tests under suitable conditions on the lower bound of the signal-to-noise ratios.

\begin{theorem}  \label{one-sample.power.1}
Let $\wt{\bSigma}_1= \wh{\bSigma}_1$ be the sample covariance matrix. Assume that either condition {\rm(M1)} holds and $p=O(n^{r/2-1-\de})$ for some $\de>0$, or condition {\rm(M2)} holds and $\log(p) = o (n^{\gamma/2} )$. For given $0<\alpha<1$, write $\la(p,\alpha)=  \s{2\log(p)} + \s{2\log(1/\alpha)}$, and let $\{ \varepsilon_{n} \}_{n\geq 1}$ be an arbitrary sequence of positive numbers satisfying $\varepsilon_{n} \rightarrow 0$ and $\varepsilon_{n} \s{\log(p)} \rightarrow \infty$ as $n \rightarrow \infty$. As ${n} , p \rightarrow \infty$, we have
{\rm(i)} $
   \mathbb{P}_{H^{\rI}_1} \{  \Psi^{\rI}_{\ns ,\alpha} =1  \} \rightarrow 1
$
if
$
   {\max_{1\leq k\leq p} |\mu_{1 k}|}/{\max_{1\leq k\leq p} \sigma_{1 k} }   \geq (1+\varepsilon_{n})
    n^{-1/2} \la(p,\alpha),
$ and {\rm (ii)} $\mathbb{P}_{H^{\rI}_1} \{ \Psi^{\rI}_{\rs ,\alpha} =1  \} \rightarrow 1$
if $\max_{1\leq k\leq p}  {|\mu_{1 k} |}/{ \sigma_{1 k} }  \geq (1+\varepsilon_{n}) n^{-1/2} \la(p,\alpha)$.
\end{theorem}

Theorem \ref{one-sample.power.1} reveals that the test based on studentized statistics is consistent in a larger testable region in comparison to the test based on non-studentized statistics. As a complement to Theorem~\ref{one-sample.size}, the asymptotic size of the proposed two-sample tests without screening is reported below.

\begin{theorem}   \label{two-sample.size}
Let $(\wt{\bSigma}_1, \wt{\bSigma}_2)= (\wh{\bSigma}_1, \wh{\bSigma}_2)$ and $\nu \in \{\ns, \rs\}$. Assume that either condition~{\rm(i)} or condition~{\rm(ii)} in Theorem~{\rm\ref{one-sample.size}} holds. Then as $n, p \rightarrow \infty$,
$
 \mathbb{P}_{H^{\rII}_0} \{  \Psi_{\nu,\alpha}^{\rII} =1 \} \rightarrow   \alpha.
$
\end{theorem}

Theorem~\ref{two-sample.size} implies that, under proper moment conditions, the proposed two-sample non-screening tests maintain nominal size $\alpha$ asymptotically, while allowing for either a polynomial or an exponential rate of growth of the dimension $p$ with respect to the sample size $n$. In Theorem~\ref{two-sample.power.1} below, the asymptotic power of the two-sample non-screening tests is analyzed.

\begin{theorem} \label{two-sample.power.1}
Let $(\wt{\bSigma}_1, \wt{\bSigma}_2)=(\widehat{\bSigma}_1, \widehat{\bSigma}_2)$. Assume that either condition {\rm(M1)} holds and $p=O(n^{r/2-1-\de})$ for some $\de>0$, or condition {\rm(M2)} holds and $\log(p) = o ( n^{\gamma /2  } )$. For given $0<\alpha<1$, let $ \la(p,\alpha)$ and $\{ \varepsilon_n\}_{n\geq 1}$ be as in Theorem {\rm\ref{one-sample.power.1}}. As $n, p \rightarrow \infty$, we have
{\rm(i)} $\mathbb{P}_{H^{\rII}_1} \{   \Psi^{\rII}_{\ns,\alpha}  =  1 \}  \rightarrow  1$ if
$
{  \max_{1\leq k\leq p} |\mu_{1 k}-\mu_{2 k}| }/{ \max_{1\leq k\leq p} ( \sigma_{1 k}^2/n + \sigma_{2 k}^2/m )^{1/2}} \geq  (1+ \varepsilon_n )   \la(p,\alpha)$, and {\rm (ii)} $
\mathbb{P}_{H^{\rII}_1}\{   \Psi^{\rII}_{\rs ,\alpha} = 1 \} \rightarrow 1 $
if
$
    \max_{1\leq k\leq p}  {|\mu_{1 k}-\mu_{2 k}|}/{(  \sigma_{1 k}^2/n + \sigma_{2 k}^2/m  )^{1/2} }  \geq (1+\varepsilon_n)  \la(p,\alpha).
$
\end{theorem}

%
%
%

The following theorem establishes asymptotic properties of the proposed two-step testing procedures. Part (i) in Theorem \ref{two-step.one-sample.theory} below shows that the type I error of the proposed screening-based two-step procedures can be controlled by the prescribed significance level asymptotically. Similar to the comparison between the studentized and non-studentized tests in Theorem \ref{one-sample.power.1}, parts (ii) and (iii) in Theorem \ref{two-step.one-sample.theory} below also imply that the screening-based two-step studentized test is consistent in a larger region than its non-studentized counterpart.

\begin{theorem}  \label{two-step.one-sample.theory}
Let $\wt{\bSigma}_1= \wh{\bSigma}_1$. Assume that either condition {\rm(M1)} holds and $p=O(n^{r/2-1-\de})$ for some $\de>0$, or condition {\rm(M2)} holds for some $\gamma\geq \frac{1}{2}$ and $\log(p) = o (n^{1/7} )$. We have
{\rm (i)} $\limsup_{n\rightarrow \infty}    \mathbb{P}_{H^{\rI}_0} \{  \Psi^{f,\rI}_{\nu ,\alpha} =1  \} \leq  \alpha$ for $\nu\in\{\ns,\rs\}$,
{\rm (ii)} $\mathbb{P}_{H^{\rI}_1}  \{ \Psi^{f,\rI}_{\ns ,\alpha} =1 \} \rightarrow 1$ if the condition for part {\rm(i)} in Theorem {\rm\ref{one-sample.power.1}} holds, {\rm (iii)}
$\mathbb{P}_{H^{\rI}_1}  \{ \Psi^{f,\rI}_{\rs ,\alpha} =1 \} \rightarrow 1$ if the condition for part {\rm(ii)} in Theorem {\rm\ref{one-sample.power.1}} holds.
\end{theorem}

Similarly, the following theorem establishes the limiting null property and the asymptotic power for the proposed two-step procedures with pre-screening in the two-sample settings.

\begin{theorem} \label{two-step.two-sample.theory}
Let $(\wt{\bSigma}_1, \wt{\bSigma}_2)=(\widehat{\bSigma}_1, \widehat{\bSigma}_2)$. Assume that either condition {\rm(M1)} holds and $p=O(n^{r/2-1-\de})$ for some $\de>0$, or condition {\rm(M2)} holds for some $\gamma\geq \frac{1}{2}$ and $\log(p) = o ( n^{1/7  } )$. We have {\rm (i)} $\limsup_{n\rightarrow \infty}    \mathbb{P}_{H^{\rII}_0} \{  \Psi^{f,\rII}_{\nu ,\alpha} =1  \} \leq  \alpha$ for $\nu\in\{\ns,\rs\}$,
{\rm (ii)} $\mathbb{P}_{H^{\rII}_1}  \{ \Psi^{f,\rII}_{\ns ,\alpha} =1 \} \rightarrow 1$ if the condition for part {\rm(i)} in Theorem {\rm\ref{two-sample.power.1}} holds, and {\rm (iii)}
$\mathbb{P}_{H^{\rII}_1}  \{ \Psi^{f,\rII}_{\rs ,\alpha} =1 \} \rightarrow 1$ if the condition for part {\rm(ii)} in Theorem {\rm\ref{two-sample.power.1}} holds.
\end{theorem}

\section{Simulation studies}\label{section:numerical}


In this section, we report the simulation results from several experiments to evaluate the performance of the proposed tests, including the non-studentized test without screening $\Psi_{\ns,\alpha}$, the studentized test without screening $\Psi_{\rs,\alpha}$, the non-studentized test with screening $\Psi_{\ns,\alpha}^f$ and the studentized test with screening $\Psi_{\rs,\alpha}^f$, for both one- and two-sample problems. For ease of exposition, we suppress the superscripts $\rm (I)$ and $\rm (II)$. To demonstrate the proposed tests, we also implemented peer testing procedures for comparison. For the one-sample problem, we compared the proposed tests with the test by \cite{ZhongChenXu_2013} (denoted by ZCX hereafter) and the Higher Criticism (HC) procedure by \cite{DonohoJin04} . We used the method proposed by \cite{Li_Siegmund_2015} to obtain more accurate approximation of the critical values in HC procedure. For the two-sample problem, we experimented the tests by \cite{ChenQin_2010} (denoted by CQ hereafter) and \cite{CaiLiuXia_2014} (denoted by CLX hereafter) as well as the HC procedure.
. 

In the simulation studies, we considered a wide range of covariance structures, including both the sparse and dense settings to investigate the numerical performance of the proposed tests. We generate data with sample sizes $n=40$ or $80$ in one-sample case and $(n,m)=(40,40)$ or $(80,80)$ in two-sample case. The dimension $p$ took values in $120,360$ or $1080$. The empirical size and power were defined as the proportion of the rejection among  $1500$ replications. We used the sample covariance matrices to generate $M=1500$ Monte Carlo samples to compute the critical values for our proposed tests. We only report the results for six models in this section and more models are considered in the supplementary material.


\subsection{One-sample case} \label{subsection:numerical one sample}

We took $\bmu_1=\textbf{0}$ under the null hypothesis, whereas, under the alternative, we took $\bmu_1=(\mu_{11},\ldots, \mu_{1p})^\T$ to have $\lfloor \kappa p^{r} \rfloor$ non-zero entries uniformly and randomly drawn from $\{1,\ldots,p\}$, where $\kappa$ was an integer and $\lfloor x \rfloor$ denotes the integer part of $x$. We took $r=0,0.4,0.5,0.7$ and $0.85$, where $\kappa=8$ if $r=0$ and $\kappa=1$ otherwise. The choices of $r=0$ and $r=0.7$ or $0.85$ correspond to the sparse and non-sparse settings, respectively. The magnitudes of non-zero entries $\mu_{1\ell}$ were set to be $\{ 2\beta\sigma_{1,\ell\ell}\log(p)/n \}^{1/2}$, where $\sigma_{1,\ell\ell}$ denotes the $\ell$th diagonal entry of $\bSigma_1$. We took $\beta=0.01,0.2,0.4,0.6$ and use $\beta=0.01$ to mimic the scenario of weak signals.

The following two models were used to generate random samples $\bX_i=\bZ_i+\bmu_1$ for $ i=1,\ldots, n$, where $\{\bZ_i\}_{i=1}^n\stackrel{\textrm{i.i.d}}{\sim}\textrm{N}(\bzero,\bSigma_1)$ with $\bSigma_1=(\sigma_{1,k \ell})_{1\leq k,\ell \leq p}$.
\begin{itemize}

\item Model 1$^{\rm{(I)}}$: $\sigma_{1,k \ell}=0.4^{|k-\ell |}$ for $1\leq k,\ell \leq p$. 


\item Model 2$^{\rm{(I)}}$: 
Let $\{\theta_k\}_{k=1}^p\stackrel{\textrm{i.i.d.}}{\sim}{\rm Unif}(1,2)$. We took $\sigma_{1,kk}=\theta_k$ and $\sigma_{1,k \ell } =\rho_\alpha(|k -\ell |)$ for $k \neq \ell$, where
$\rho_\alpha(e)=\tfrac12 \{(e+1)^{2H}+(e-1)^{2H}-2e^{2H}\}$ with $H=0.9$.
\end{itemize}
Model 1$^{\rm{(I)}}$ has sparse covariance structure while Model 2$^{\rm{(I)}}$ takes long range dependence into account which exhibits a non-sparse structure. In addition, we considered the following model with non-Gaussian data to study the robustness of the proposed tests against Gaussian assumptions. The covariance structure in the following Model 3$^{\rm{(I)}}$ is non-sparse.

\begin{itemize}
\item Model 3$^{\rm{(I)}}$: 
Let $\{\bX_i\}_{i=1}^n\stackrel{\textrm{i.i.d.}}{\sim} t_\omega(\bmu_1,\bSigma_1) $, where $t_\omega(\bmu_1,\bSigma_1)$ is the non-central multivariate $t$-distribution with non-central parameter $\bmu_1$, degrees of freedom $\omega=5$, and $\sigma_{1,k \ell}=0.995^{|k-\ell|}$.

\end{itemize}

Simulation results for the tests $\Psi_{\ns,\alpha}$, $\Psi_{\rs,\alpha}$, $\Psi_{\ns,\alpha}^f$ and $\Psi_{\rs,\alpha}^f$ and the ZCX and HC tests are summarized in Table \ref{size1} and Figure \ref{f01}. Table \ref{size1} displays the empirical sizes of all the tests. It can be seen that in all the models, the empirical sizes of the non-studentized tests $\Psi_{\ns,\alpha}$ and $\Psi_{\ns,\alpha}^f$ are reasonably close to the nominal level $0.05$ for both $n=40$ and $n=80$. The proposed studentized tests $\Psi_{\rs,\alpha}$ and $\Psi_{\rs,\alpha}^f$ have slightly inflated size when $n$ is relatively small but improve with larger sample sizes. The ZCX test maintains
the nominal size for Model 1$^{\rm{(I)}}$ but fails in the presence of long range dependence or non-sparse covariance structures. The HC procedure also fails in maintaining the nominal significance when the sample size $n$ is small or the dependency is strong and complex.

\begin{table}[h!]
    \centering
    {\renewcommand{\arraystretch}{1.1}
 \begin{tabular}{cccccccccc}\toprule
  &  \multicolumn{3}{c}{Model 1$^{\rm{(I)}}$}         &    \multicolumn{3}{c}{Model 2$^{\rm{(I)}}$}       &  \multicolumn{3}{c}{Model 3$^{\rm{(I)}}$}                                    \\ \toprule
$\text{tests}~$/$~p$ &  120&360&1080&120&360&1080&120&360&1080\\\toprule & \multicolumn{9}{c}{$n=40$}     \\ \midrule
 $\Psi_{\ns,\alpha}$  & 0.037&0.027&0.021& 0.025&0.028&0.023 &0.054&0.044&0.033\\ [0.5ex]
  $\Psi_{\rs,\alpha}$  & 0.133 & 0.126 & 0.168&0.093&0.113&0.202 & 0.065&0.080&0.096 \\ [0.5ex]
   $\Psi^f_{\ns,\alpha}$  & 0.044& 0.045& 0.043& 0.039&0.027&0.039 & 0.054&0.046&0.033  \\ [0.5ex]
    $\Psi^f_{\rs,\alpha}$  &  0.150 & 0.154& 0.194 &0.095&0.170&0.218 & 0.060&0.058&0.093 \\ [0.5ex]
ZCX  & 0.064&0.078 & 0.089& 1&1&1 & 0.382&0.487&0.673 \\ [0.5ex]
HC & 0.123& 0.225& 0.316&0.129&0.249&0.320&0.274&0.377&0.468\\
 \midrule

& \multicolumn{9}{c}{$n=80$}     \\ \midrule

 $\Psi_{\ns,\alpha}$ & 0.037&0.036&0.029&0.040&0.032&0.042 & 0.049&0.047&0.040\\ [0.5ex]
  $\Psi_{\rs,\alpha}$  & 0.060& 0.082& 0.092&0.082&0.083&0.094 &0.058&0.058&0.067 \\ [0.5ex]
   $\Psi^f_{\ns,\alpha}$  & 0.048& 0.045& 0.043&0.051&0.045&0.040 &  0.049&0.048&0.044 \\ [0.5ex]
    $\Psi^f_{\rs,\alpha}$  & 0.086& 0.097& 0.094&0.095&0.091&0.110 &0.060&0.058&0.069 \\ [0.5ex]
ZCX  & 0.080&0.072&0.071&1&1&1&  0.404&  0.506&0.702\\ [0.5ex]
HC & 0.063& 0.119& 0.142 &0.079&0.145&0.175&0.267&0.363&0.471
  \\ \bottomrule
  \end{tabular}
  }
  \caption{Empirical sizes of the proposed tests (non-studentized without screening $\Psi_{\ns,\alpha}$, studentized without screening $\Psi_{\rs,\alpha}$, non-studentized with screening $\Psi_{\ns,\alpha}^f$, and studenzied with screening $\Psi_{\rs,\alpha}^f$) for the one-sample problem \eqref{eq:onesample}, along with those of the tests by \citet{ZhongChenXu_2013} (ZCX),
   and \citet{DonohoJin04} (HC) at $5\%$  nominal significance.
    Models with Gaussian data and sparse or long range dependence (non sparse) covariance matrices, and the autoregressive model with $t$-distributed innovations are considered when $n=40, 80$ and $p=120,360,1080$. }\label{size1}
\end{table}

To compare the empirical powers, we took $n=80$ and $p=1080$. For Model 1$^{\rm{(I)}}$, we compared the proposed tests with the ZCX test (column (a) in Figure \ref{f01}), whereas, for the other two models, we only focused on comparing the four proposed tests as they maintain the nominal size reasonably well and other tests fail in size control. Column (a) in Figure \ref{f01} shows that $\Psi_{\rs,\alpha}$, $\Psi_{\rs,\alpha}^f$ and $\Psi_{\ns,\alpha}^f$ provide non-trivial powers against alternatives with sparse signals ($r=0$) even under the weak signal settings ($\beta=0.01$); in contrast, the ZCX test improves its power as the signal getting dense, which is expected for sum of squares-type statistics. As the signal strength increases, all tests under consideration gain powers. The proposed tests with screening, $\Psi_{\ns,\alpha}^f$ and $\Psi_{\rs,\alpha}^f$, outperform the ZXC test under sparse alternatives ($r=0,0.4$), and their powers are close to that of the ZCX test for dense signals ($r\geq 0.7$). From columns (b) and (c) in Figures \ref{f01}, we observe that the screening procedure substantially improves the power performance of the tests for all settings, which reflects the heuristic discussions and motivations in Section \ref{ps:sec1}. The non-studentized test with screening $\Psi_{\ns,\alpha}^f$ performs comparably to, or better than, the studentized test without screening $\Psi_{\rs,\alpha}$ under sparse alternatives ($r\leq 0.5$). This suggests that $\Psi_{\ns,\alpha}^f$ is more preferable in practice given its capability in maintaining the nominal significance for small sample size. 
\begin{figure}[h!]
\centering
\begin{tabular}{ccc}  
\includegraphics[width=1.75in,height=4.25cm]{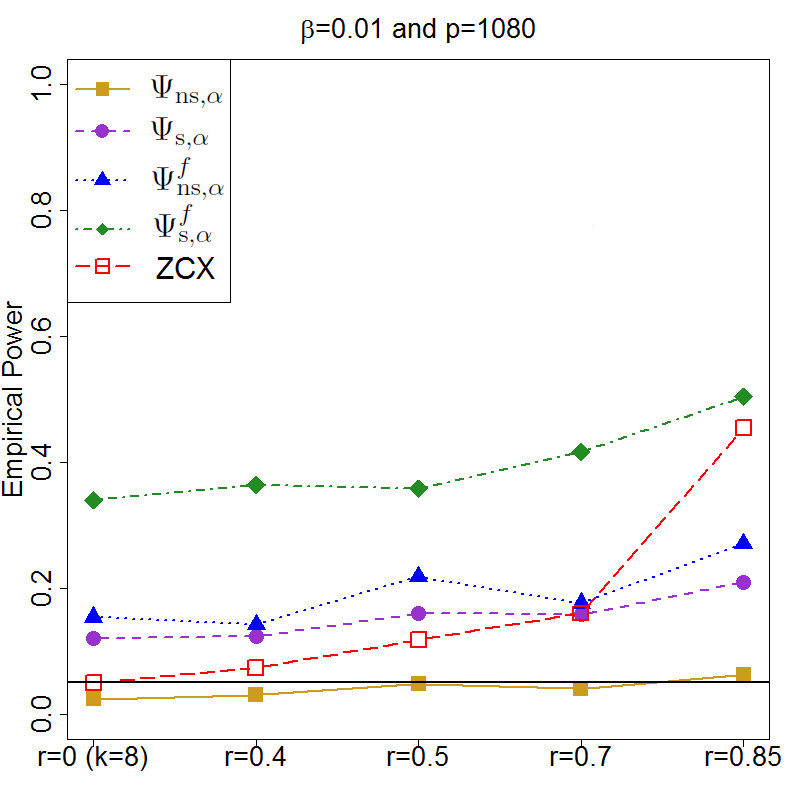}
&\includegraphics[width=1.75in,height=4.25cm]{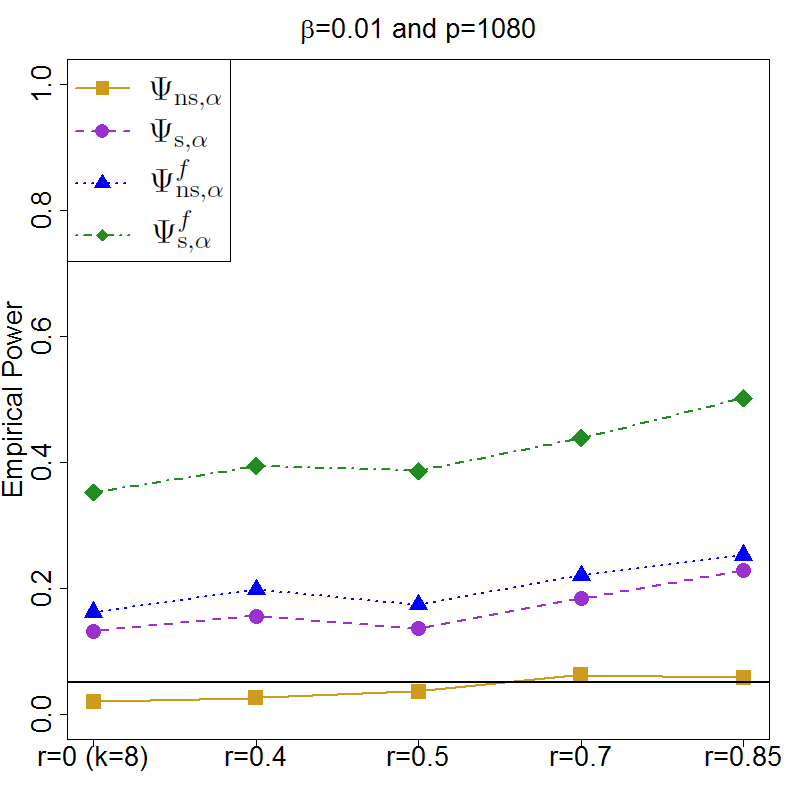}
&\includegraphics[width=1.75in,height=4.25cm]{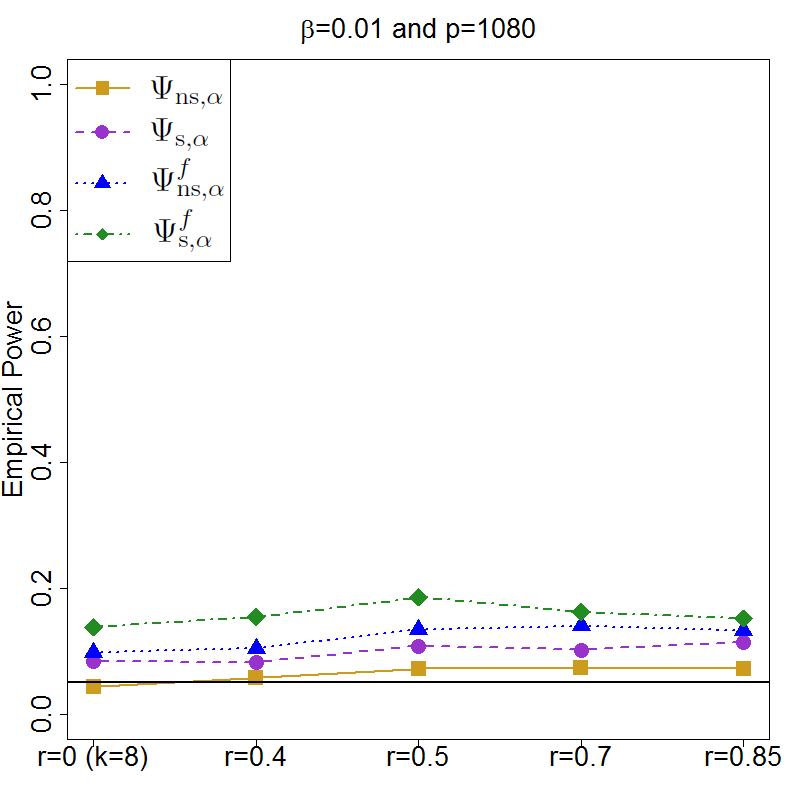}\\
\includegraphics[width=1.75in,height=4.25cm]{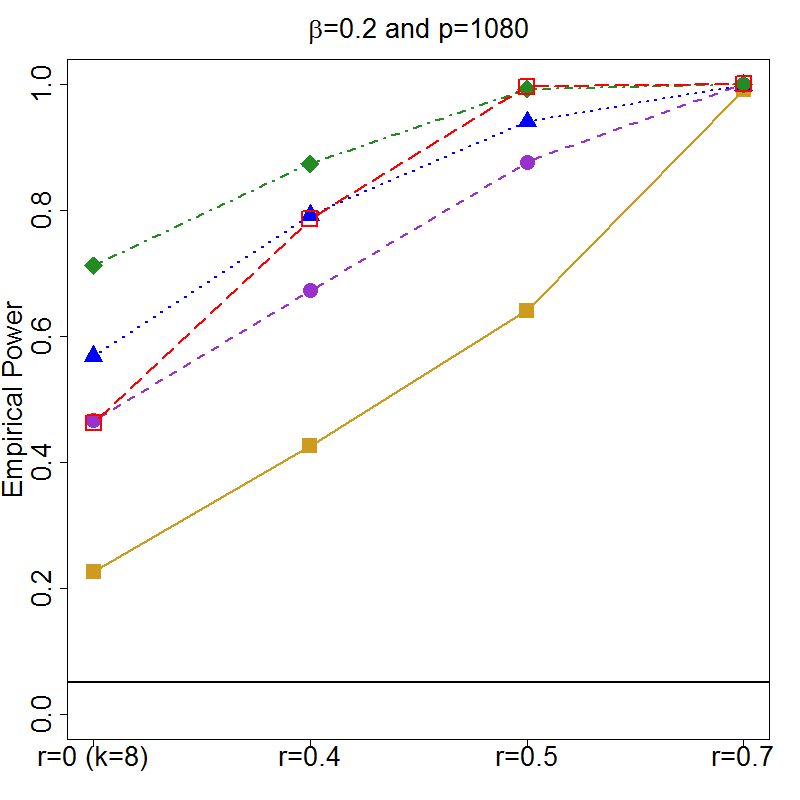}
&\includegraphics[width=1.75in,height=4.25cm]{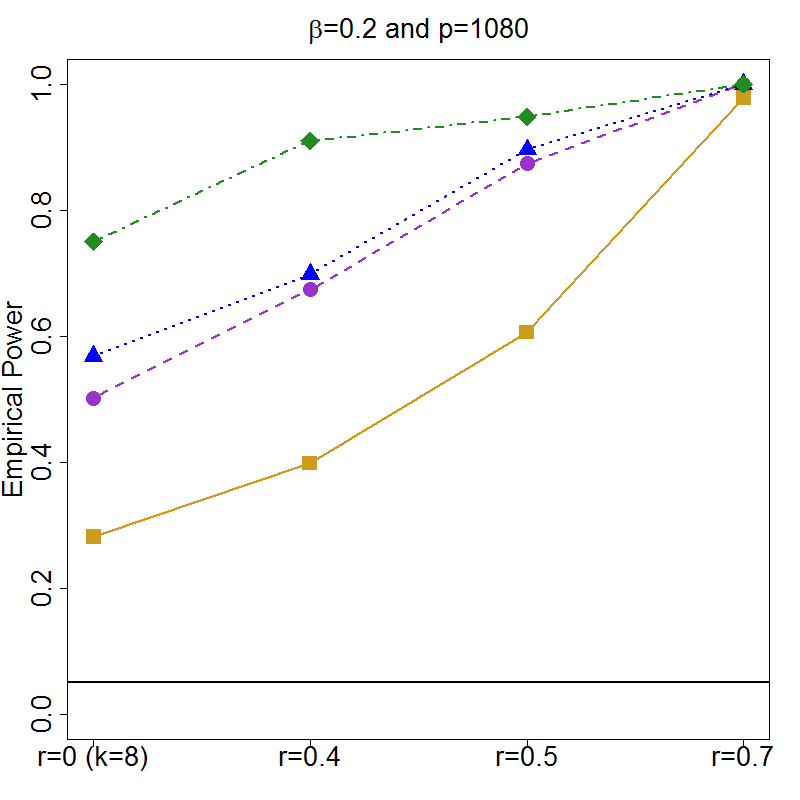}
&\includegraphics[width=1.75in,height=4.25cm]{M3S1p1080f2.png}\\
\includegraphics[width=1.75in,height=4.25cm]{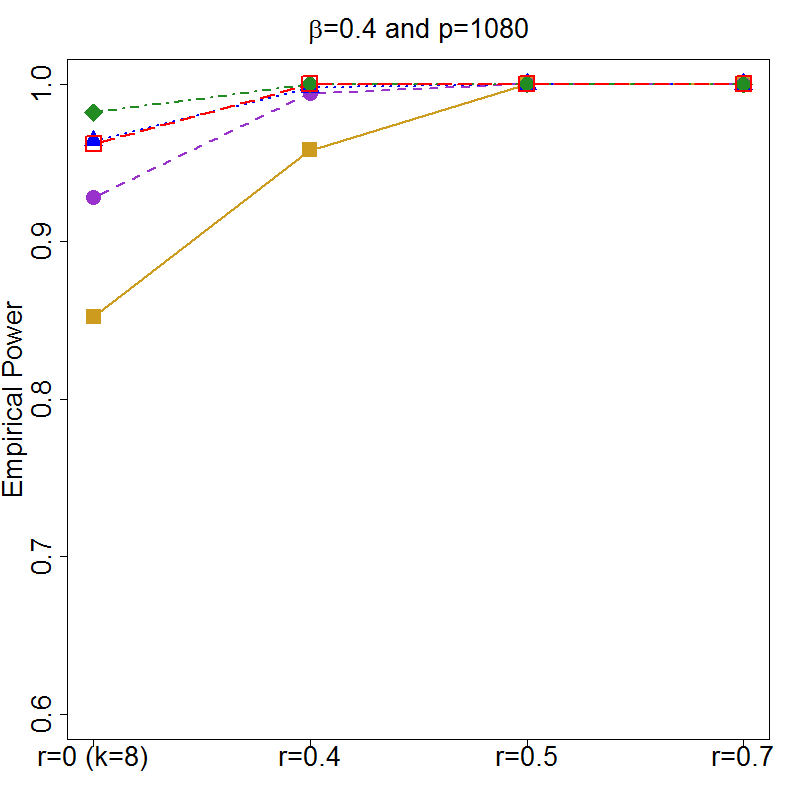}
&\includegraphics[width=1.75in,height=4.25cm]{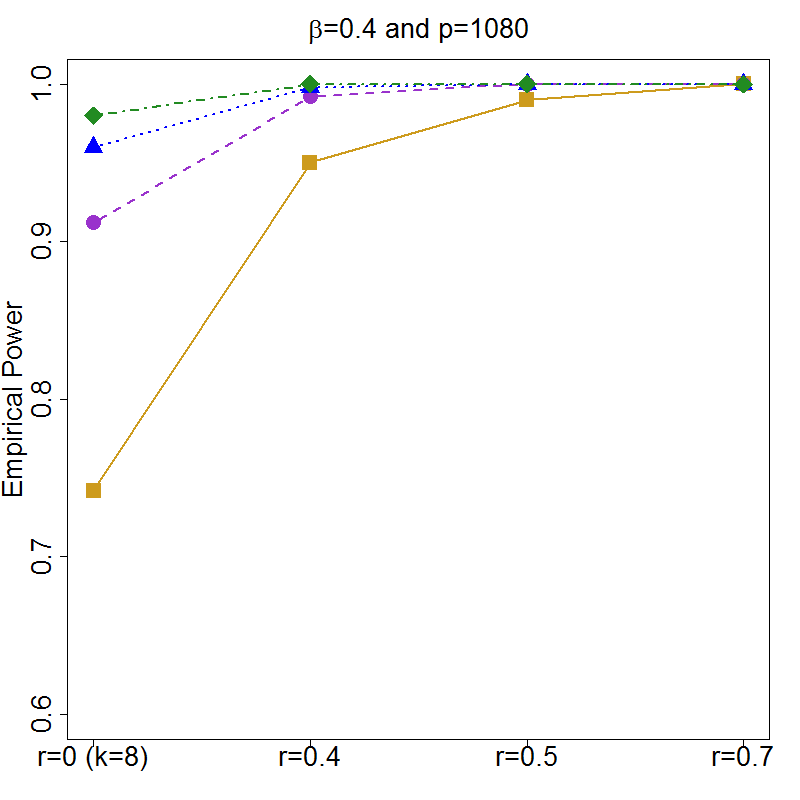}
&\includegraphics[width=1.75in,height=4.25cm]{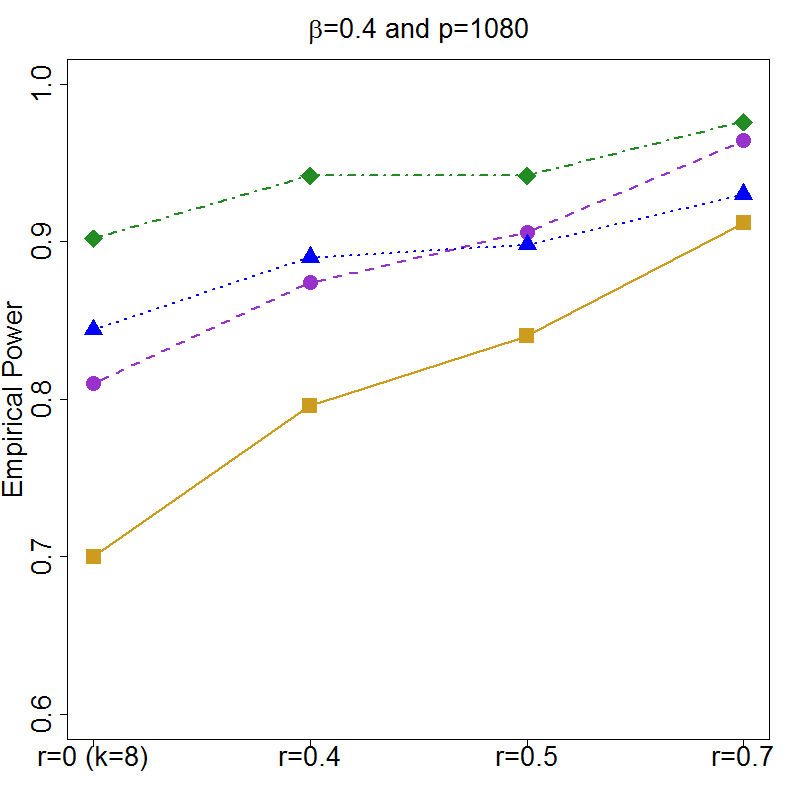}\\
\includegraphics[width=1.75in,height=4.25cm]{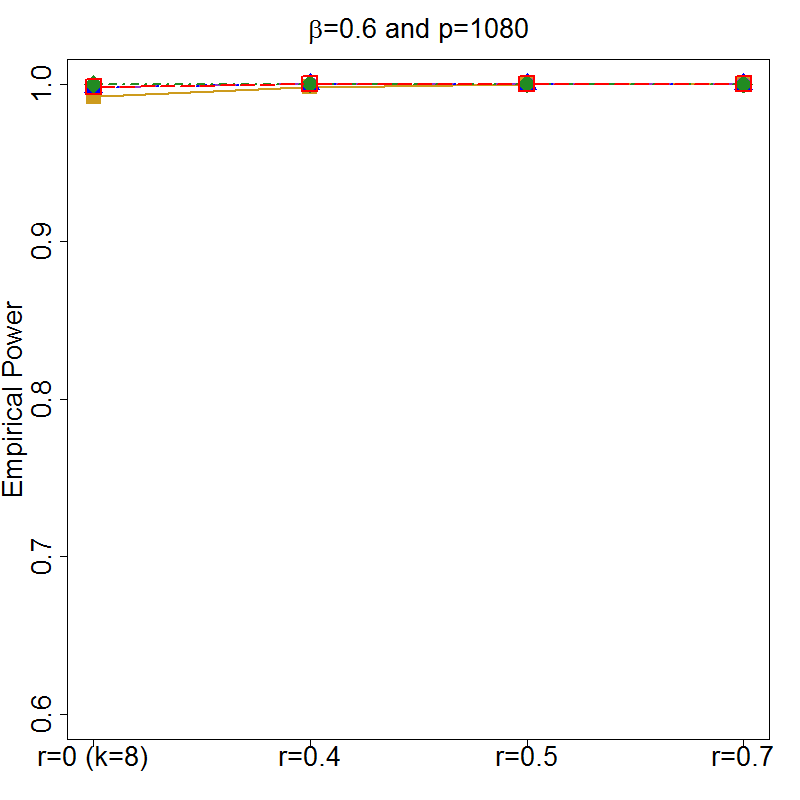}
&\includegraphics[width=1.75in,height=4.25cm]{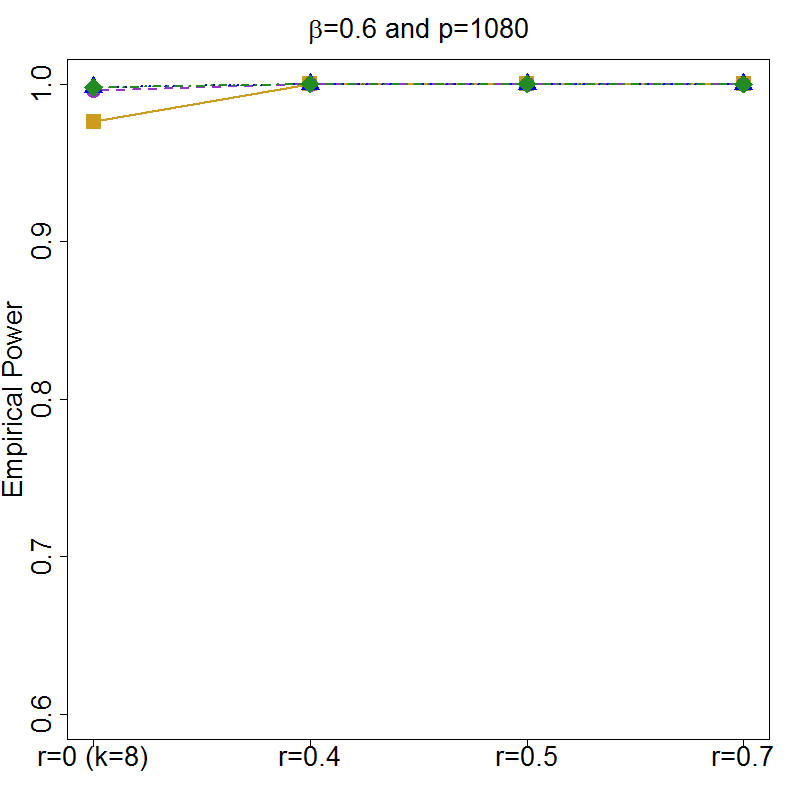}
&\includegraphics[width=1.75in,height=4.25cm]{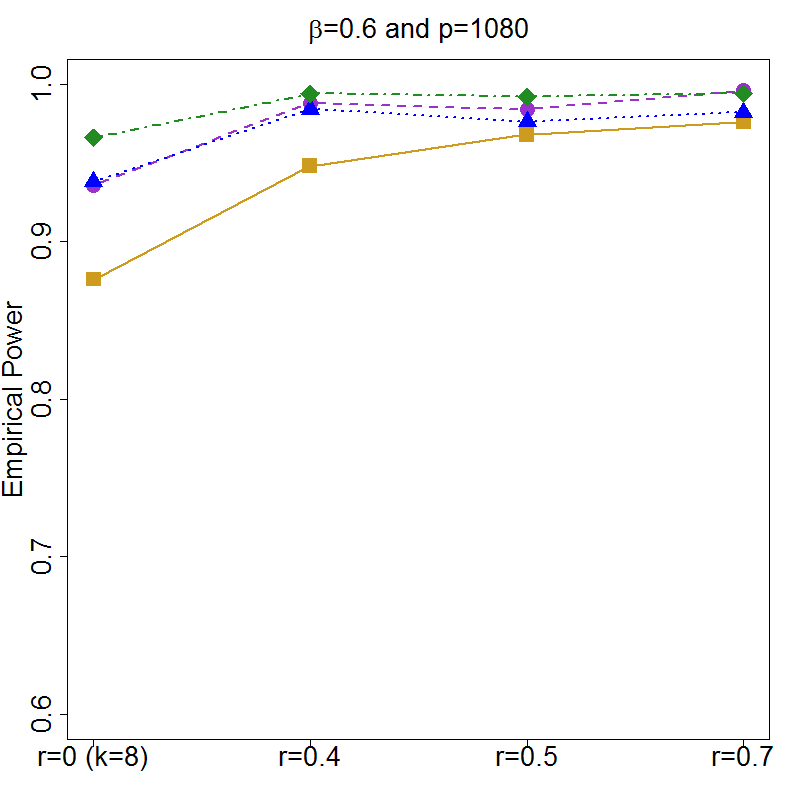}\\
(a) Model 1$^{\rm(I)}$&(b) Model 2$^{\rm (I)}$&(c) Model 3$^{\rm (I)}$
\end{tabular}
\caption{  Empirical powers of the proposed tests (non-studentized without screening $\Psi_{\ns,\alpha}$, studentized without screening $\Psi_{\rs,\alpha}$, non-studentized with screening $\Psi_{\ns,\alpha}^f$, and also studenzied with screening $\Psi_{\rs,\alpha}^f$) against alternatives with different levels of the signal strength ($\beta$) and sparsity ($1-r$) for the one-sample problem \eqref{eq:onesample} when $n=80$ and $p=1080$ at 5\% nominal significance for the Gaussian data and sparse covariance matrices in Model 1$^{\rm{(I)}}$ (column (a)), the Gaussian data and long range dependence covariance matrices in Model 2$^{\rm{(I)}}$ (column (b)), and the autoregressive process model, Model 3$^{\rm{(I)}}$, with $t$-distributed innovations (column (c)). Column (a) also displays the powers of the test by \citet{ZhongChenXu_2013} (ZCX).}
   \label{f01}
\end{figure}

\subsection{Two-sample case}
\label{subsection:numerical two sample}

We took $\bmu_1=\bmu_2=\textbf{0}$ under the null hypothesis, whereas, under the alternative, we let $\bmu_1=(\mu_{11},\ldots, \mu_{1p})^\T$ to have $\lfloor \kappa p^{r} \rfloor$ non-zero entries uniformly and randomly drawn from $\{1,\ldots,p\}$, where $\kappa$ is an integer. As before, we considered $r=0,0.4,0.5,0.7$ and $0.85$, where $\kappa=8$ if $r=0$ and $\kappa=1$ otherwise. 
The magnitudes of non-zero entries $\mu_{1\ell}$ were set to be $\{2\beta\sigma_{\ell\ell}\log(p)(1/n+1/m)\}^{1/2}$, where $\sigma_{\ell\ell}$ is the $\ell$th diagonal entry of the pooled covariance matrix $\bSigma_{1,2}$ as in \eqref{Sigma12}. We took $\beta=0.01,0.2,0.4,0.6$.

The following two models were used to generate random samples $\bX_i=\bZ_{1,i}+\bmu_1, \bY_j=\bZ_{2,j}+\bmu_2$ for $ i=1,\ldots, n$ and $j=1,\ldots, m$, where $\{\bZ_{1,i}\}_{i=1}^n\stackrel{\textrm{i.i.d.}}{\sim}\textrm{N}(\bzero,\bSigma_1)$ and $\{\bZ_{2,j}\}_{j=1}^m\stackrel{\textrm{i.i.d.}}{\sim}\textrm{N}(\bzero,\bSigma_2)$ with $\bSigma_1=(\sigma_{1,k \ell})_{1\leq k,\ell \leq p}$ and $\bSigma_2=(\sigma_{2,k \ell})_{1\leq k,\ell \leq p}$, respectively.
\begin{itemize}
\item Model 1$^{\rm{(II)}}$: For $k=1,\ldots, p$ and $q=1,2$, $\sigma_{q,kk}\stackrel{\textrm{i.i.d.}}{\sim}{\rm Unif}(2,3)$, $\sigma_{q,k\ell}=0.7$ for $10(t-1)+1\leq k\neq \ell\leq 10t$, where $t=1,\ldots, \lfloor p/10\rfloor$, and $\sigma_{q,k\ell}=0$ otherwise.

\item Model 2$^{\rm{(II)}}$: Let $\bF=(f_{k\ell})_{1\leq k,\ell\leq p}$ with $f_{k k}=1,f_{k,k+1}=f_{k+1,k}=0.5$, $\bU_q \sim \mathcal{U}(\mathcal{V}_{p,k_0})$, the uniform distribution on the Stiefel manifold for $q=1,2$, and $\bTheta=\textrm{diag}\{\theta_{11},\ldots,\theta_{pp}\}$ with $\theta_{kk}\stackrel{\textrm{i.i.d.}}{\sim}{\rm Unif}(1,6)$. Set $k_0=10$ and put $\bSigma_q=\bTheta^{1/2}(\bF+\bU_q \bU_q^\T)\bTheta^{1/2}$ for $q=1,2$.

\end{itemize}

Model 1$^{\rm{(II)}}$ and Model 2$^{\rm{(II)}}$ are with sparse and non-sparse covariance structures, respectively.  In addition, we considered the following model with non-Gaussian data.

\begin{itemize}
\item Model 3$^{\rm{(II)}}$: 
Let $\{\bX_i\}_{i=1}^n\stackrel{\textrm{i.i.d.}}{\sim} t_{\omega_1}(\bmu_1,\bSigma_1)$ and $\{\bY_j\}_{j=1}^m\stackrel{\textrm{i.i.d.}}{\sim} t_{\omega_2}(\bmu_2,\bSigma_2)$, where $\omega_1=5, \omega_2=7$, $\sigma_{1,k \ell}=0.995^{|k-\ell|}$ and $\sigma_{2,k \ell}=0.7^{|k-\ell|}$.

\end{itemize}

The numerical results on the proposed tests $\Psi_{\ns,\alpha}$, $\Psi_{\rs,\alpha}$, $\Psi_{\ns,\alpha}^f$ and $\Psi_{\rs,\alpha}^f$ and the HC, CQ and CLX tests are summarized in Table \ref{size2} and Figure \ref{f02}. Table \ref{size2} displays the empirical sizes. It can be seen that in all the models, the empirical sizes for  $\Psi_{\ns,\alpha}$ and $\Psi_{\ns,\alpha}^f$ are reasonably close to the nominal level $0.05$ for both $(n,m)=(40,40)$ and $(80,80)$. The studentized tests, $\Psi_{\rs,\alpha}$ and $\Psi_{\rs,\alpha}^f$, have slightly inflated significance when the sample size is relatively small but improve when the sample size increases. Additionally, the CLX test fails to maintain the nominal size for Model 3$^{{\rm (II)}}$ due to the strong dependency in the covariance structures. Analogous to the observation in Section \ref{subsection:numerical one sample}, it is difficult for the HC procedure to maintain the nominal significance when the sample size is small or the dependency is strong and complex. The CQ test maintains the nominal significance reasonably well in all the models.

\begin{table}[h!]
    \centering
    {\renewcommand{\arraystretch}{1.1}
   \begin{tabular}{cccccccccc}\toprule

  &  \multicolumn{3}{c}{Model 1$^{(\rm{II})}$}      &      \multicolumn{3}{c}{Model 2$^{(\rm{II})}$}    &      \multicolumn{3}{c}{Model 3$^{(\rm{II})}$}                    \\ \toprule

$\text{tests}~$/$~p$ &  120&360&1080&120&360&1080&120&360&1080\\midrule
     & \multicolumn{9}{c}{$(n,m)=(40,40)$}     \\ \midrule

 $\Psi_{\ns,\alpha}$ & 0.039&0.041&0.041&0.042&0.044&0.039& 0.052&0.036&0.042\\ [0.5ex]
  $\Psi_{\rs,\alpha}$  &  0.094&0.112&0.125 &0.092&0.097&0.116&0.086&0.090&0.092 \\ [0.5ex]
   $\Psi^f_{\ns,\alpha}$  & 0.055&0.048&0.057& 0.049&0.055&0.054 &  0.055&0.039&0.052 \\ [0.5ex]
    $\Psi^f_{\rs,\alpha}$  & 0.092&0.120&0.152& 0.098&0.131&0.053 & 0.090&0.094&0.094 \\ [0.5ex]
HC&0.086&0.156&0.157&0.078&0.144&0.148&0.172&0.237&0.283\\ [0.5ex]
CQ & 0.044&0.049&0.034& 0.046&0.049&0.051& 0.064&0.066&0.054  \\[0.5ex]
CLX & 0.101&0.103&0.138&0.081&0.087&0.098 &  0.204&0.181&0.137\\
 \midrule

& \multicolumn{9}{c}{$(n,m)=(80,80)$}     \\ \midrule

 $\Psi_{\ns,\alpha}$ & 0.054&0.039&0.046& 0.053&0.040&0.040& 0.046&0.045&0.047\\ [0.5ex]
  $\Psi_{\rs,\alpha}$  & 0.074&0.062&0.086& 0.058&0.064&0.090 & 0.059&0.065&0.074\\ [0.5ex]
   $\Psi^f_{\ns,\alpha}$  & 0.065&0.052&0.060& 0.063&0.050&0.058 &  0.047&0.048&0.056 \\ [0.5ex]
    $\Psi^f_{\rs,\alpha}$  & 0.088&0.076&0.098&0.070&0.080&0.093 & 0.062&0.069&0.086 \\ [0.5ex]
HC  &0.068&0.086&0.099&0.053&0.085&0.085&0.165&0.239&0.263\\ [0.5ex]
CQ & 0.046&0.039&0.048&0.048&0.038&0.048 &  0.044&0.054&0.056\\[0.5ex]
CLX & 0.107&0.090&0.104&0.057&0.057&0.089& 0.289&0.352& 0.297
  \\ \bottomrule
  \end{tabular}
  }

  \caption{Empirical sizes of the proposed tests (non-studentized without screening $\Psi_{\ns,\alpha}$, studentized without screening $\Psi_{\rs,\alpha}$, non-studentized with screening $\Psi_{\ns,\alpha}^f$, and studenzied with screening $\Psi_{\rs,\alpha}^f$) for the two-sample problem \eqref{eq:twosample}, along with those of the tests by  \citet{DonohoJin04} (HC), \citet{ChenQin_2010} (CQ), and \citet {CaiLiuXia_2014} (CLX) at 5\% nominal significance. Models with Gaussian data and sparse or non-sparse covariance matrices, and with non-Gaussian data are considered when $n=m=40$ or $80$ and $p=120,360,1080$.  }\label{size2}

\end{table}

To evaluate the power, we compared the proposed tests with the CQ and CLX tests for $(n,m)=(80,80)$ and $p=1080$. It can be seen that the tests with screening, $\Psi_{\ns,\alpha}^f$ and $\Psi_{\rs,\alpha}^f$, outperform both the CQ and CLX tests against alternatives with sparse signals $(r=0)$ for different signal strength $\beta$. On the other hand, all the tests perform similarly when the signals become less sparse and strong. The CQ test gains more powers when signals become less sparse, as expected for sum of squares-type statistics. Its power approaches to those of the proposed tests with screening $\Psi_{\ns,\alpha}^f$ and $\Psi_{\rs,\alpha}^f$ when the signals become less sparse and stronger ($r\geq 0.5, \beta\geq 0.4$) in the models except Model 3$^{\rm{(II)}}$. In Model 3$^{\rm{(II)}}$, all the proposed tests outperform the CQ test substantially as the sum of squares-type test statistics may lose power for heavy tailed sampling distributions. The CLX test performs similarly to the  $\Psi_{\ns,\alpha}$ and $\Psi_{\rs,\alpha}$, but is outperformed by the proposed tests with screening for all settings. The simulation results agree with the heuristic discussion and the theoretical justification that the screening step substantially improves the power of proposed tests. Similar to the observations in Section \ref{subsection:numerical one sample},  $\Psi_{\ns,\alpha}^f$ is preferable in practice whenever the sample size is relatively small.

\begin{figure}[h]
\centering
\begin{tabular}{ccc}
\includegraphics[width=1.75in,height=4.25cm]{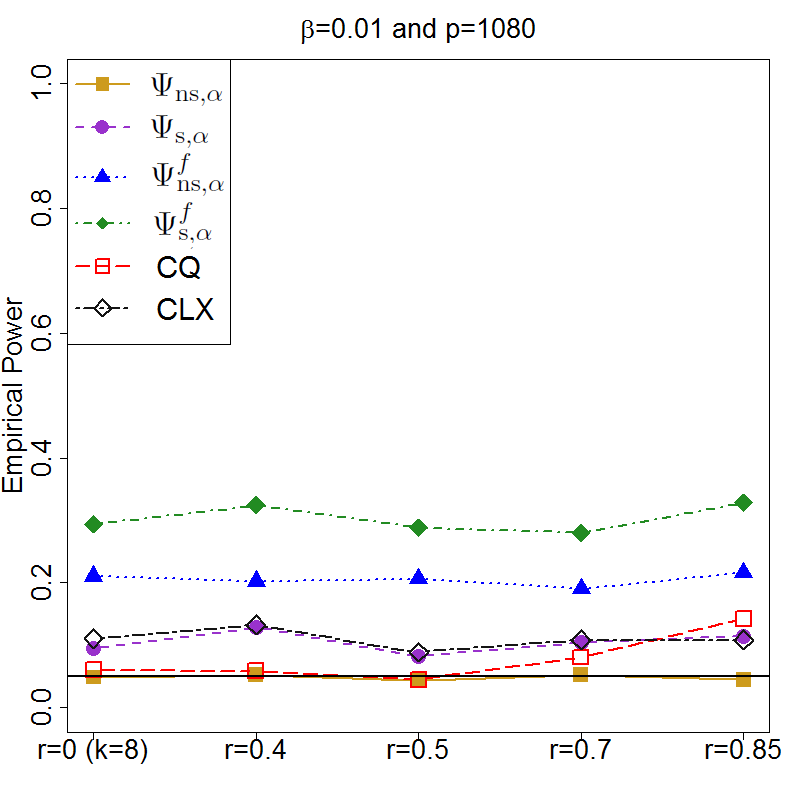}
&\includegraphics[width=1.75in,height=4.25cm]{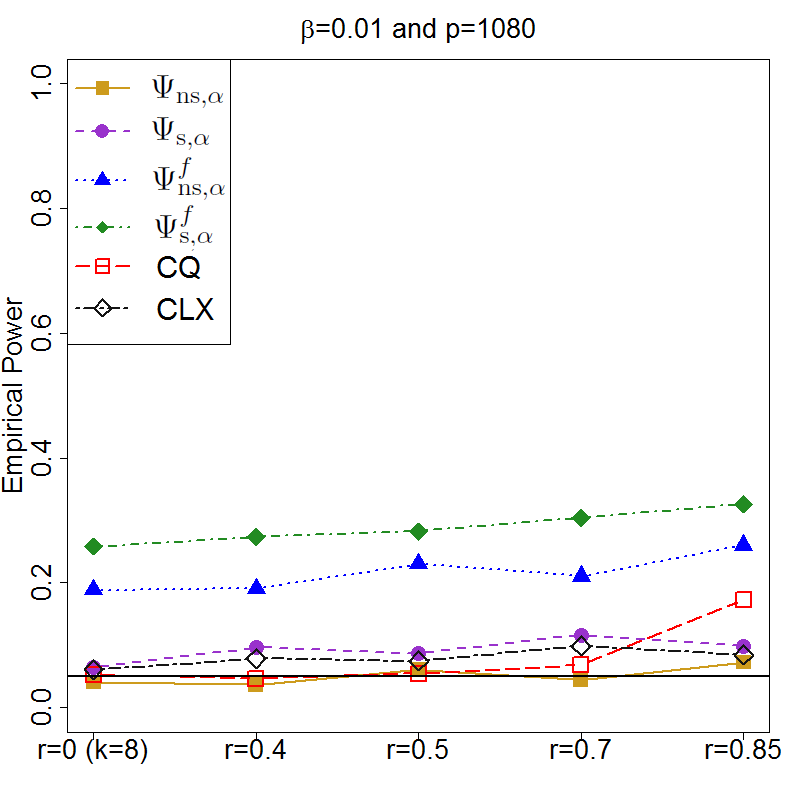}
&\includegraphics[width=1.75in,height=4.25cm]{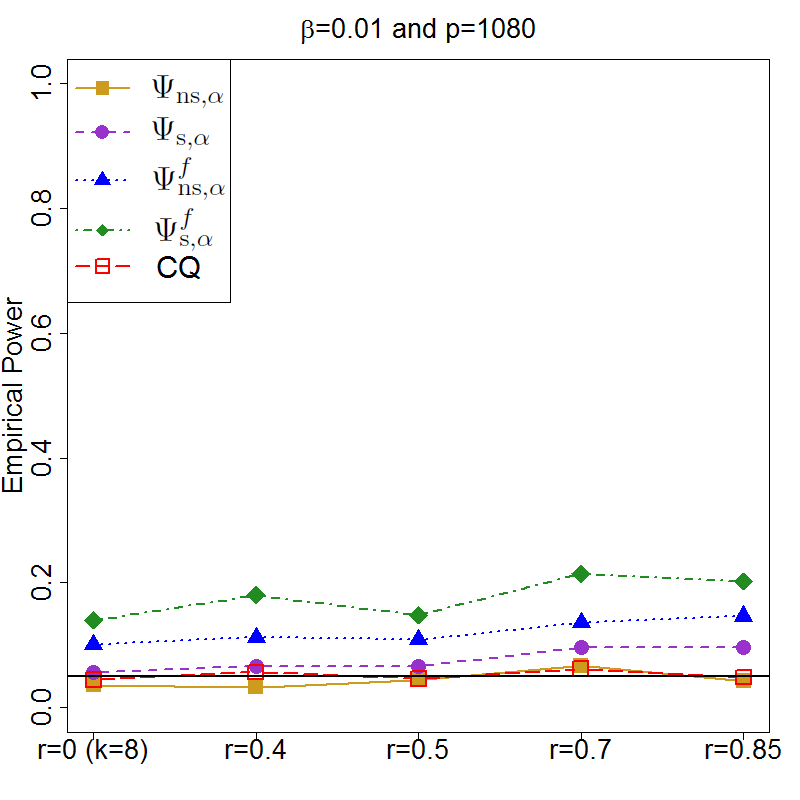}\\
\includegraphics[width=1.75in,height=4.25cm]{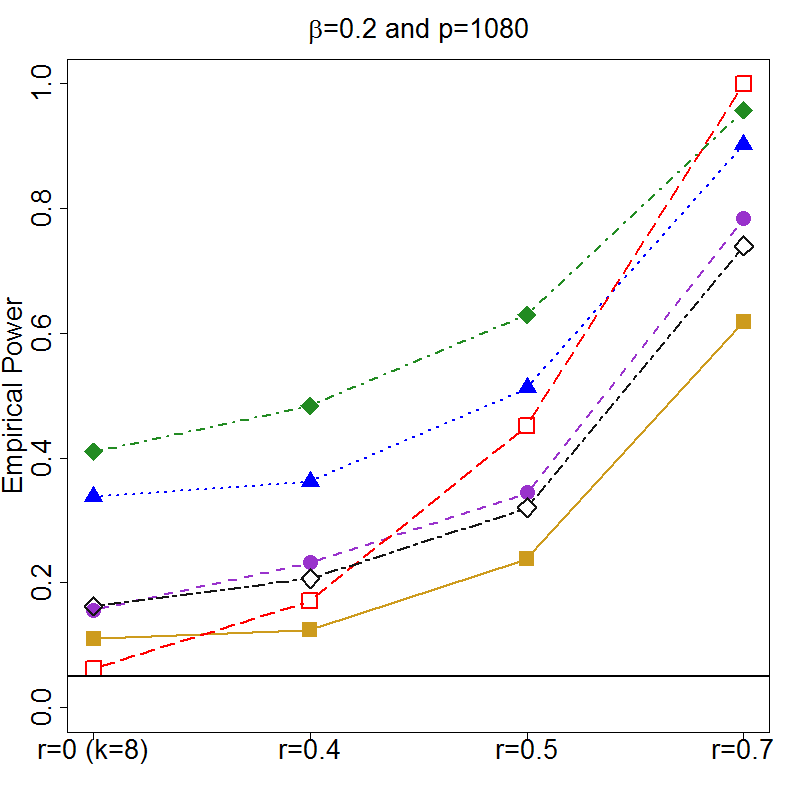}
&\includegraphics[width=1.75in,height=4.25cm]{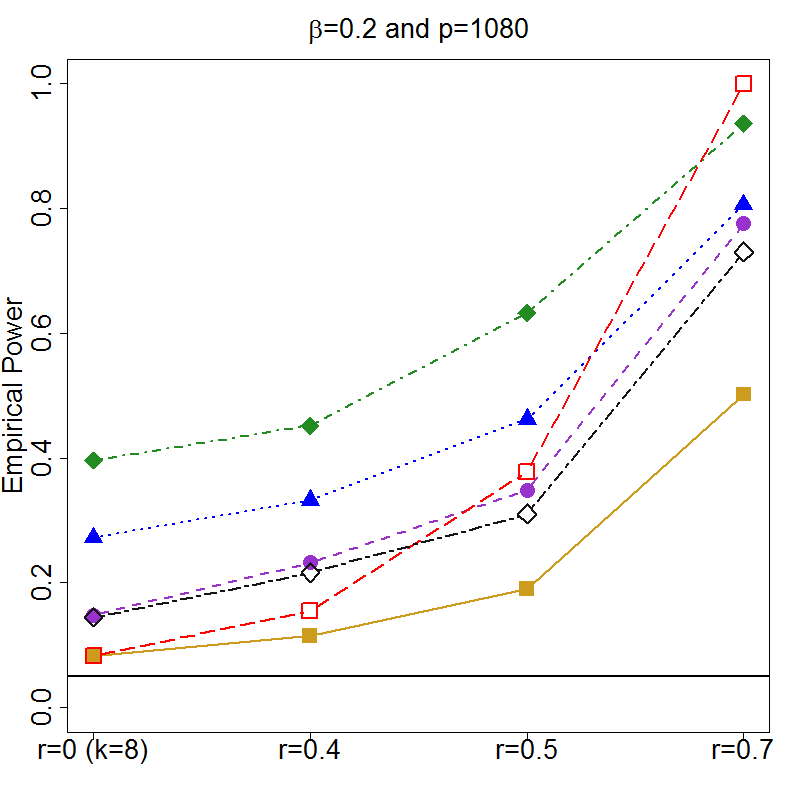}
&\includegraphics[width=1.75in,height=4.25cm]{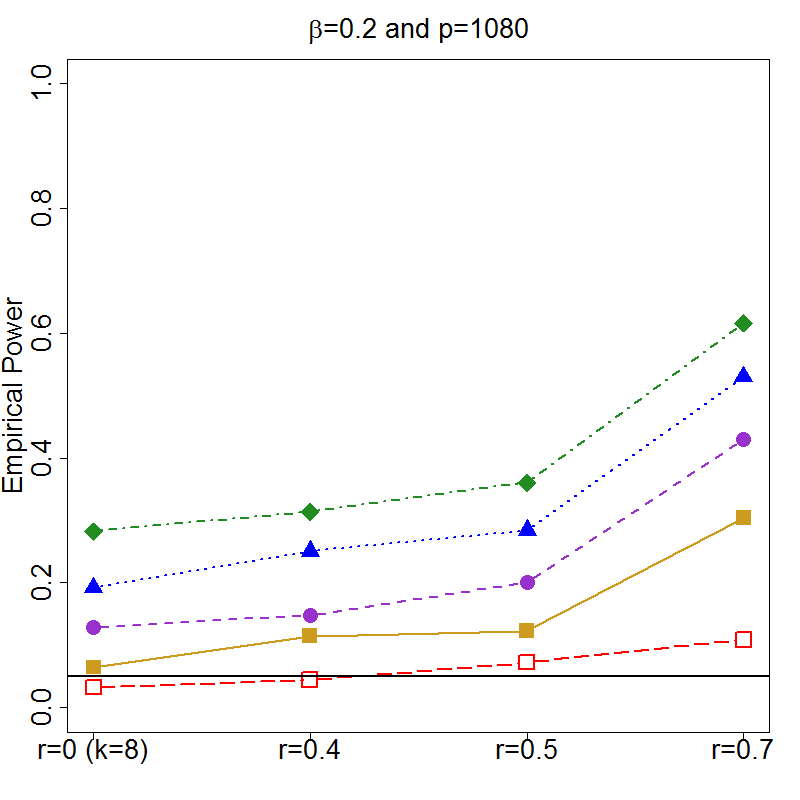}\\
\includegraphics[width=1.75in,height=4.25cm]{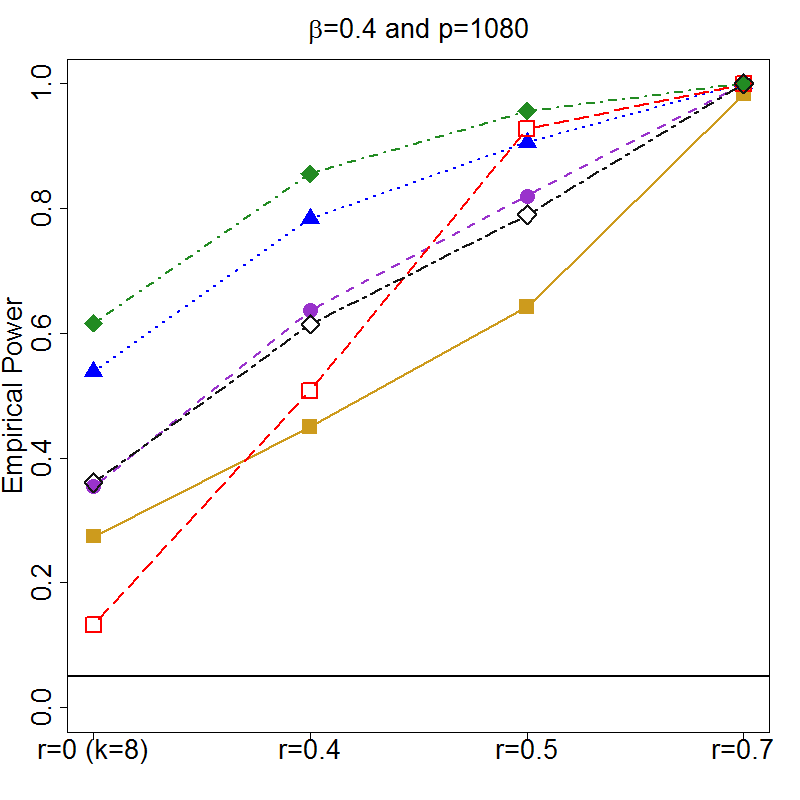}
&\includegraphics[width=1.75in,height=4.25cm]{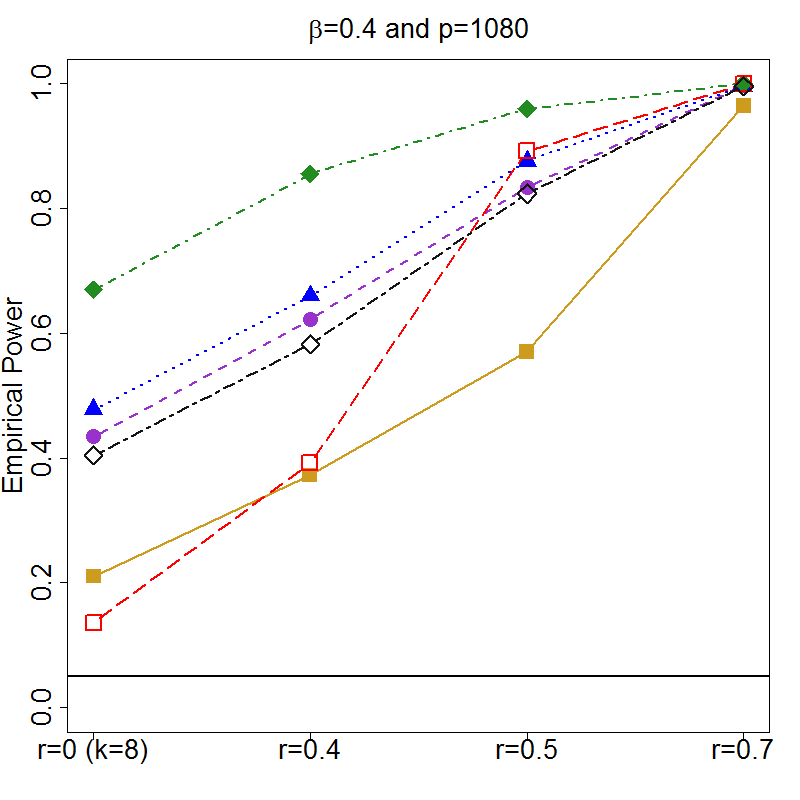}
&\includegraphics[width=1.75in,height=4.25cm]{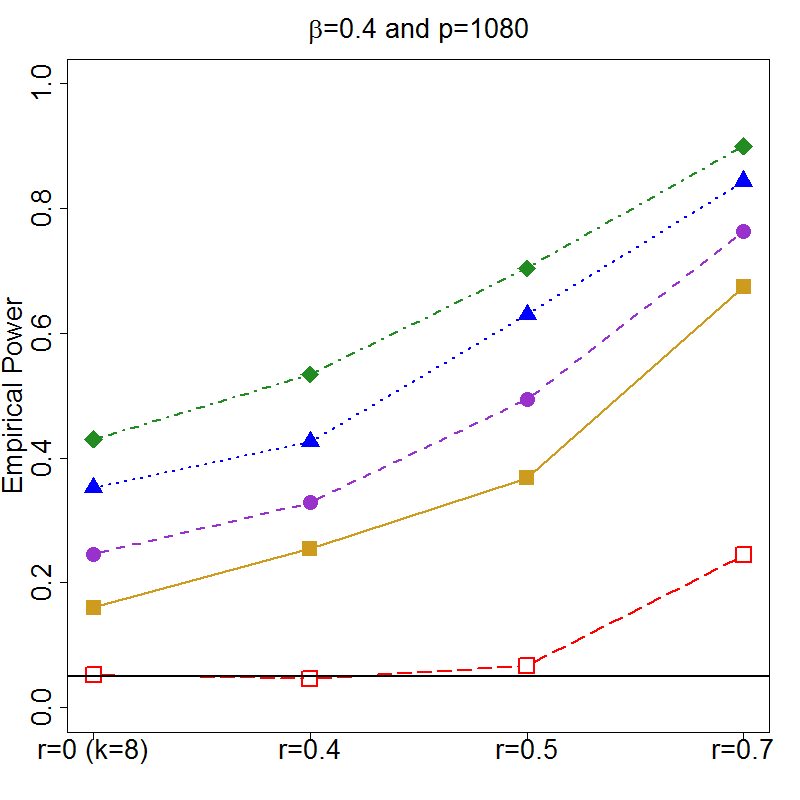}\\
\includegraphics[width=1.75in,height=4.25cm]{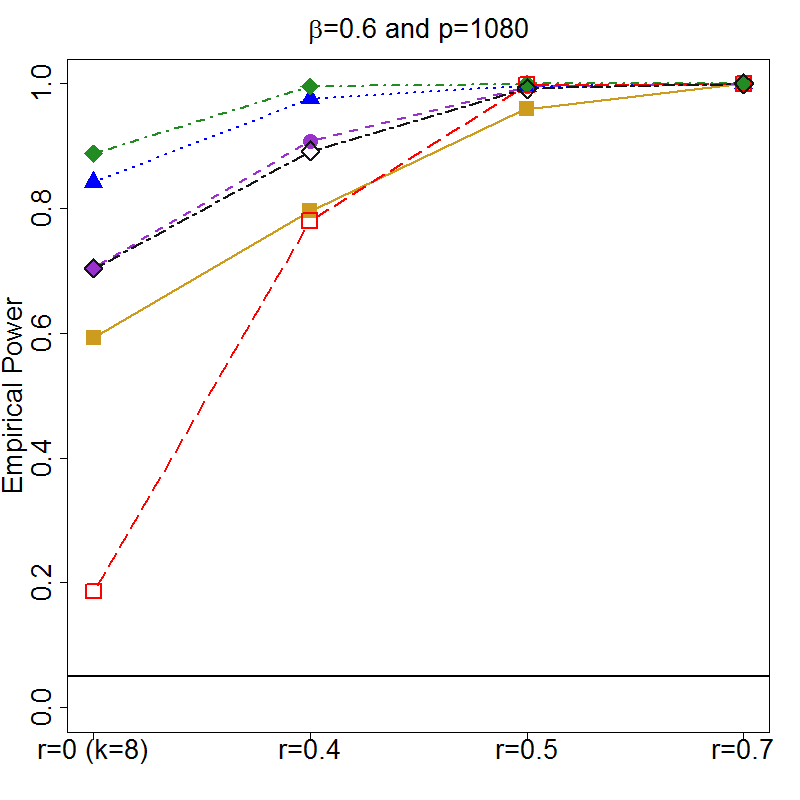}
&\includegraphics[width=1.75in,height=4.25cm]{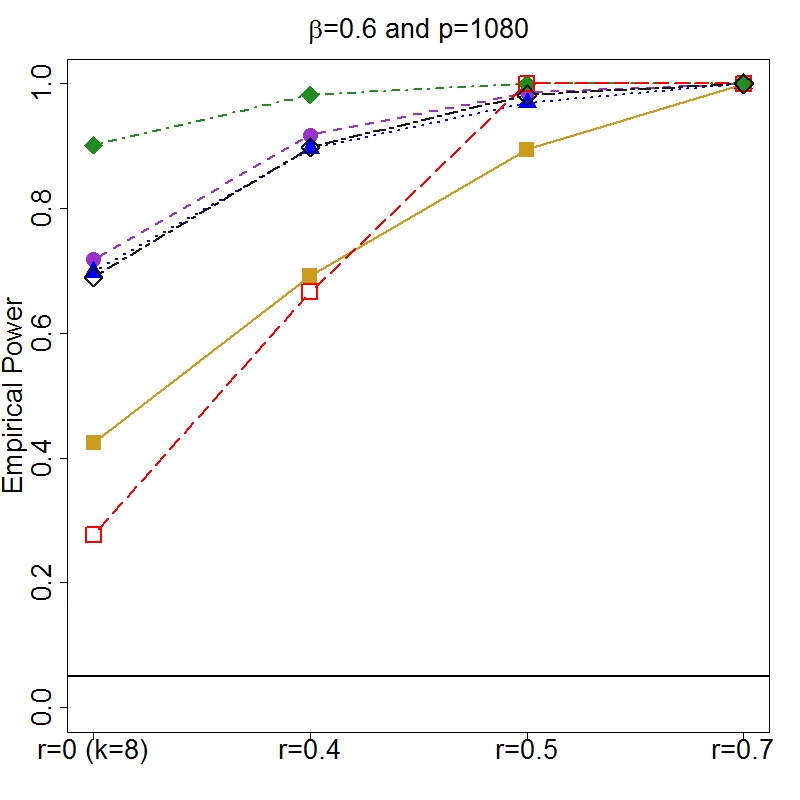}
&\includegraphics[width=1.75in,height=4.25cm]{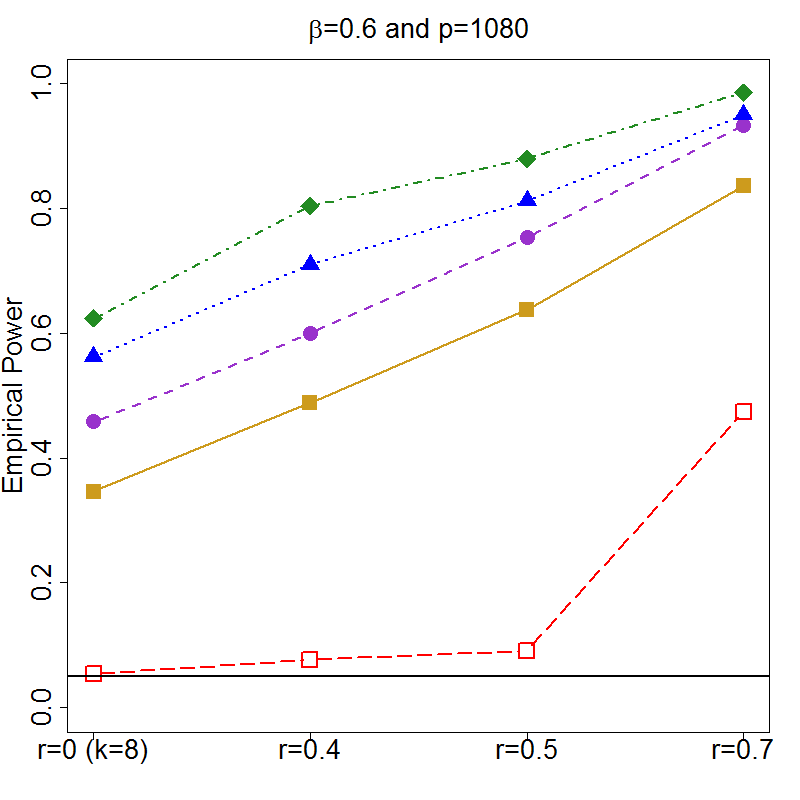}\\
(a) Model 1$^{\rm(II)}$&(b) Model 2$^{\rm (II)}$&(c) Model 3$^{\rm (II)}$

\end{tabular}
\caption{  Empirical powers of the proposed tests (non-studentized without screening $\Psi_{\ns,\alpha}$, studentized without screening $\Psi_{\rs,\alpha}$, non-studentized with screening $\Psi_{\ns,\alpha}^f$, and also studenzied with screening $\Psi_{\rs,\alpha}^f$) against alternatives with different levels of the signal strength ($\beta$) and sparsity ($1-r$) for the two-sample problem \eqref{eq:twosample} when $n=80$ and $p=1080$ at 5\% nominal significance for the Gaussian data and sparse covariance matrices in Model 1$^{\rm{(II)}}$ (column (a)), the Gaussian data and non-sparse covariance matrices in Model 2$^{\rm{(II)}}$ (column (b)), and the non-Gaussian data in Model 3$^{\rm{(II)}}$ (column (c)). The powers of the tests by \citet{ChenQin_2010} (CQ) and \citet{CaiLiuXia_2014} (CLX) are also displayed.}
   \label{f02}
\end{figure}

In summary, the numerical results show that the proposed tests, particularly the studentized tests  and the non-studentized test with screening, $\Psi_{\rs,\alpha}$, $\Psi_{\rs,\alpha}^f$ and  $\Psi_{\ns,\alpha}^f$, outperform the existing methods when the covariance structure is non-sparse and complex. The proposed tests are robust against both unknown covariance structures and Gaussianity. The $\Psi_{\ns,\alpha}^f$ maintains the nominal significance for small sample sizes and has good powers against sparse alternatives, which is recommended for practical applications with relatively small sample size. The $\Psi_{\textrm{s},\alpha}^f$ is more powerful and thus is preferable in applications with relatively large samples, such as biomedical research with a large cohort.

More extensive simulations were carried out for dimensions $p=120$ and $360$, from which the comparisons are consistent with the cases that are reported here. The empirical powers of all the tests also increase in $p$. All the additional simulation results are placed in the online supplementary materials. Furthermore, extra simulations were reported in the supplementary materials to demonstrate that the proposed procedures may benefit from using regularized covariance estimations when the covariance matrices do admit special structures.


\setcounter{equation}{0}

\section{Empirical study} \label{real data}


Analysis and interpretation based on gene-sets or GO terms derive more power than focusing on individual gene in extracting biological insights \citep{S05}. It has drawn increasing attentions to identify GO terms associated with biological states of interest 
\citep{S05,ET_2007,RNR_2008}. 
A particular GO term belongs to one of the three categories of gene ontologies of interest: biological processes (BP), cellular components (CC) and molecular functions (MF).

Statistically, identifying interesting gene-sets out of $G$ candidate gene-sets $\mathcal{S}_1,\ldots, \mathcal{S}_G$ based on independent samples from two biological states ($q=1,2$) is equivalent to test hypotheses
$H_{0s}: \bmu_{1,s}=\bmu_{2,s}$ versus $H_{1s}: \bmu_{1,s}\neq \bmu_{2,s}$ for $s=1,\ldots,G$, where $\bmu_{q,s}$ models the mean expression levels of $p_s$ genes in the gene-set $\mathcal{S}_s$ under biological state $q$. It is common that gene-sets overlap with each other as one particular gene may belong to several functional groups, and the size of a gene-set $p_s$ usually range from a small to a very large number. The selection of gene-sets therefore encounters both multiplicity and high dimensionality. Similar to \cite{ChenQin_2010}, we applied the proposed tests to each gene-set. With $p$-values obtained for all $G$ gene-sets, we further employed the multiple testing methods such as the Benjamini-Yekutieli (BY) procedure \citep{BY01} for controlling the false discovery rate (FDR) under dependeny to identify significant gene-sets. 

We applied the above procedure to a human acute lymphoblastic leukemia (ALL) dataset which is available at \url{http://www.ncbi.nlm.nih.gov}. The data contains gene expression levels from microarray experiments for patients suffering from ALL of either T-lymphocyte type or B-lymphocyte type leukemia. This dataset was originally analyzed by \cite{C04} to provide insight into the genetic mechanism on ALL development and it was also analyzed by \cite{DKV08} and \cite{ChenQin_2010} using different methodologies. To illustrate the proposed tests, we focus on the 75 patients of B-lymphocyte type leukemia, who were classified into two groups: 35 patients with BCR/ABL fusion and 40 patients with cytogenetically normal NEG, i.e., $n=35$ and $m=40$. We employed the approach in \cite{G2005} to conduct preliminary data processing. To focus on high dimensional scenarios, we also excluded gene-sets with $p_s\le 19$. It remained $G=1853,262$ and $284$ unique GO terms in the BP, CC and MF categories, respectively. And the largest gene-set contained $p_s=3050, 3145$ and $3040$ genes in the BP, CC and MF categories, respectively. Given the complexity of the data processing and collection procedures, batch effects may exist and result in unreliable results. Therefore, we further employ the surrogate variable analysis (SVA) method proposed by \cite{Leek_Storey_2007} to remove the potential batch effects and other unwanted variations in the data. In summary, two surrogate variables were found by SVA and removed from the original ALL expression data. Identifications of gene-sets associated to the BCR/ABL fusion display biological insights on the development of B-lymphocyte type leukemia and provide lists of functional groups for potential clinical treatments. We aim to identify gene-sets with significantly different expression levels between the BCR/ABL and NEG groups for each of the three categories.

The sample size of the ALL data is relatively small comparing to the maximum $p_s$, we therefore employed the proposed two-sample non-studentized tests $\Psi_{\ns,\alpha}$ and $\Psi_{\ns,\alpha}^f$ in the analysis as suggested by simulation studies in Section \ref{section:numerical}. Based on empirical $p$-values, we further employed the BY procedure for controlling the FDR at $0.015$ and identify significant gene-sets. For the proposed tests, we let $M=50000$ and used the sample covariance matrices to generate samples. Simulation studies in Section \ref{section:numerical} have shown that the test by \cite{CaiLiuXia_2014} may inflate type I error rate for small sample size, we therefore only consider the test by \cite{ChenQin_2010} (CQ) as a reference. For each category, the numbers of gene-sets being identified are summarized in Table \ref{t2}. All the gene-sets identified by the proposed two-step test $\Psi_{\ns,\alpha}^f$ are also identified by CQ
methods. This suggests that CQ test may over-detect some  disease-associated gene-sets. Moreover, $\Psi_{\ns,\alpha}^f$ found more disease associated gene-sets than $\Psi_{\ns,\alpha}$, which reflects the power improvement of the proposed two-step testing procedure as discussed before.
\begin{table}[h]
\centering \noindent{
\begin{tabular}{c|c|ccc|cccc}\toprule
GO     & \multirow{2}{*}{$\Psi_{\ns,\alpha}$} &  \multicolumn{3}{c|}{$\Psi_{\ns,\alpha}^f$ and CQ}   & \multirow{2}{*}{Total} &  \multirow{2}{*}{$\max_{s} p_s$} & \multirow{2}{*}{$\min_s p_s$} &
\multirow{2}{*}{$\lfloor \bar{p}_s\rfloor$} \\[.75ex]
  Category    &     & $\Psi_{\ns,\alpha}^f$ only & Both & CQ only  \\
 \midrule

BP &  601& 0 & 956 & 560& 1853 &  3050 & 20 & 150\\[0.5ex]
CC & 52 & 0 & 99 & 17   &  262 & 3145 &  19 & 280\\[0.5ex]
MF & 95 & 0 & 150 & 77  &  284 & 3040 & 19 & 157\\
\bottomrule
\end{tabular}}
\caption{Numbers of identified BCR/ABL associated gene-sets for each GO category using different tests in conjunction with the BY procedure by \citet{BY01} for controlling FDR at $0.015$. Columns labeled by the name of tests records the number of identified gene-sets by the corresponding testing procedures, where $\Psi_{\ns,\alpha}$ and $\Psi_{\ns,\alpha}^f$ are the proposed non-studentized tests without and with screening, and CQ stands for the test by \citet{ChenQin_2010}.}\label{t2}
\end{table}

By carefully investigating the gene-sets identified by both the proposed tests $\Psi_{\ns,\alpha}$ and $\Psi_{\ns,\alpha}^f$, we found that gene-sets GO:0005758 (mitochondrial intermembrane space) and GO:0004860 (protein kinase inhibitor activity) were identified as diseases-associated in the CC and MF categories. The functions of these two interesting gene-sets were recently studied and recognized associated with the development of ALL \citep{BK14,CW09}. Particularly, the protein kinase inhibition has been considered to be essential for the mechanism of T-lymphocyte type ALL \citep{CW09} and our finding suggests its connection with B-lymphocyte type ALL as well. The association of these gene-sets with the ALL may deserve further biological validations using the polymerase chain reaction.

\section{Supplementary Materials}
Web Appendices, which include proofs of the main theorems and additional numerical results referenced in Section \ref{section:theory} and \ref{section:numerical} are available with this paper at the Biometrics website on Wiley Online Library.
\vspace{-0.5cm}

\section*{Acknowledgement}
The authors thank the Co-Editor, the AE and two anonymous referees for
constructive comments and suggestions which have improved
the presentation of the article. Jinyuan Chang was supported in part by the Fundamental Research Funds for the Central Universities of China (Grant No. JBK150501), NSFC (Grant No. 11501462), and the Center of Statistical Research and the Joint Lab of Data Science and Business Intelligence at Southwestern University of Finance and Economics. Wen Zhou was supported in part by NSF
Grant IIS-1545994.

\vspace{-0.5cm}


\label{lastpage}
\end{document}